\documentclass[preprint,11pt]{article}
\usepackage{fullpage}
\ifdefined\usebigfont

\usepackage{times}
\usepackage[fontsize=13pt]{scrextend}
\usepackage[left=1.56in,right=1.56in,top=1.74in,bottom=1.74in]{geometry}
\else
\fi

\usepackage{amssymb,amsfonts,amsmath,amsthm,amscd,dsfont,mathrsfs}
\usepackage{graphicx,float,psfrag,epsfig,amssymb}
\usepackage[usenames,dvipsnames,svgnames,table]{xcolor}
\definecolor{darkgreen}{rgb}{0.0,0,0.9}
\usepackage[pagebackref,letterpaper=true,colorlinks=true,pdfpagemode=none,citecolor=OliveGreen,linkcolor=BrickRed,urlcolor=BrickRed,pdfstartview=FitH]{hyperref}
\usepackage{wrapfig}
\usepackage{relsize}
\usepackage{color}
\usepackage{pict2e}
\usepackage[tight]{subfigure}
\usepackage{caption}
\usepackage{nameref}
\usepackage{makecell}
\usepackage[font={small}]{caption} 
\usepackage{float}
\usepackage{bbm}
\usepackage{tikz}
\usepackage{multirow}
\usepackage{adjustbox}
\usepackage{csvsimple,longtable,booktabs}
\usepackage{graphicx}
\usepackage{cases}
\usepackage{mathtools}
\usepackage{dcolumn,booktabs}
\newcolumntype{d}[1]{D{.}{.}{#1}}

\usepackage{enumitem}
\usepackage{multirow}
\usepackage{lscape}
\usepackage{gensymb}
\usepackage{natbib}

\let\chapter\section
\usepackage[linesnumbered, ruled, vlined]{algorithm2e}
\usepackage[noend]{algorithmic}

\DeclareMathAlphabet{\mathpzc}{OT1}{pzc}{m}{it}

\newtheorem{propo}{Proposition}[section]
\newtheorem{lemma}[propo]{Lemma}

\newtheorem{thm}[propo]{Theorem}

\newtheorem*{remark}{Remark}

\def\code#1{\texttt{#1}}

\def\w{{\boldsymbol{w}}}
\def\x{{\boldsymbol{x}}}

\def\u{{\boldsymbol{u}}}
\def\e{{\boldsymbol{e}}}
\def\bbeta{{\boldsymbol{\beta}}}
\def\balpha{{\boldsymbol{\alpha}}}
\def\btheta{{\boldsymbol{\theta}}}

\def\bSigma{{\boldsymbol{\Sigma}}}
\def\P{{\mathbb{P}}}

\def\supp{\textrm{supp}}

\def\xnew{\boldsymbol{x}_{\textrm{new}}}

\def\ynew{y_{\textrm{new}}}
\def\znew{z_{\textrm{new}}}

\def\hsigma{{\hsigma}}
\def\SMI{\tilde{\boldsymbol{I}}_{k}^{\beta}(\hbtheta)}
\def\PMI{\boldsymbol{I}_{k}^{\beta}(\btheta)}
\def\SPMI{\tilde{\boldsymbol{I}}_{k}^{\beta}(\btheta)}
\renewcommand{\P}{\mathbb P}

\newcommand{\onenorm}[1]{\left\|#1\right\|_{1}}
\newcommand{\twonorm}[1]{\left\|#1\right\|_{2}}
\newcommand{\infnorm}[1]{\left\|#1\right\|_{\infty}}
\newcommand{\qnorm}[1]{\left\|#1\right\|_{q}}
\newcommand{\infMnorm}[1]{\left|#1\right|_{\infty}}

\def\E{\mathbb{E}}
\def\eps{\varepsilon}
\def\u{\boldsymbol{u}}
\newcommand{\<}{\langle}
\renewcommand{\>}{\rangle}
\def\est{\eta_{{\sf est}}}
\def\hbtheta{\widehat{\btheta}}
\def\hbbeta{\widehat{\bbeta}}
\def\hbalpha{\widehat{\balpha}}
\def\hsigma{\widehat{\sigma}}
\def\reals{\mathbb{R}}
\def\event{\mathcal{E}}
\def\z{\boldsymbol{z}}
\def\bxi{\boldsymbol{\xi}}
\def\dist{\stackrel{d}{\Rightarrow}}
\def\tQ{\tilde{Q}}
\def\de{{\rm d}}
\title{Prediction Sets for High-Dimensional Mixture of Experts Models}
\author{Adel Javanmard, Simeng Shao, Jacob Bien\thanks{
    A.~Javanmard is partially supported by the Sloan Research Fellowship
in mathematics, Adobe Data Science Faculty Research Awards and the
NSF CAREER Award DMS-1844481. J.~Bien was supported in part by NSF CAREER Award DMS-1653017.}\\
Data Sciences and Operations, University of Southern California }%

\begin{document}

\maketitle

\begin{abstract}
Large datasets make it possible to build predictive models that can
capture heterogenous relationships between the response variable and
features.  The mixture of high-dimensional linear experts model
posits that observations come from a mixture of high-dimensional linear regression
models, where the mixture weights are themselves feature-dependent.
In this paper, we show how to construct valid prediction sets for
an $\ell_1$-penalized mixture of experts model in the high-dimensional
setting.  We make use of a debiasing procedure to account
for the bias induced by the penalization and propose a novel strategy for
combining intervals to form a prediction set with coverage guarantees in the mixture setting.
Synthetic examples and an application to the prediction of critical
temperatures of superconducting materials show our method to have
reliable practical performance.
\end{abstract}

\section{Introduction}\label{sec:intro}

In traditional statistics, we imagine a universal relationship between variables
that holds across an entire population; observations not
following this relationship are dismissed as outliers.  However, we
know that reality is more complex, with numerous subpopulations
likely exhibiting distinct behaviors.  As datasets grow in size, we become
better able to detect and properly model this heterogeneity.  The
mixture of regressions model \citep{quandt1978estimating} is an
important tool for extending linear regression to this
heterogeneity-aware setting.  For a random response $y\in\reals$ and a
random vector of predictors $\x\in\reals^p$, we imagine a latent subgroup
membership $z\in\{1,\ldots,K\}$ that determines the conditional
distribution of $y$ given $\x$:
\begin{align}
y|\x,z=k\sim N(\x^T\bbeta_k,\sigma_k^2).\label{eq:normal}
\end{align}

Making predictions with this model requires estimating for each
subgroup a coefficient vector $\bbeta_k$, an error variance
$\sigma_k^2$, and a group membership probability $\pi_k=\P(z=k)$.  The
mixture of experts model (MoE, \citealt{jordan1994}) is even more flexible,
allowing these group membership probabilities to depend on the
predictors as well:
\begin{align}
z|\x\sim
  \text{Multinomial}[\boldsymbol\pi(\x)]\qquad\text{with}\quad\pi_k(\x)=\frac{\exp{(\x^T\balpha_k)}}{\sum_{\ell=1}^K
  \exp{(\x^T\balpha_\ell)}}.\label{eq:multinomial}
\end{align}
This model is expressive enough to capture subpopulations that change
in prevalence depending on the conditions.  For example, \citet{hyun2020} develop a high-dimensional mixture of
experts approach to modeling phytoplankton subpopulations as a
function of environmental covariates in the ocean.  The (log) diameter of
the phytoplankton cells within each specific subpopulation are taken to
be Gaussian with mean depending on environmental covariates (expressed through nonzero
values of $\bbeta_1,\ldots,\bbeta_K$); however, the prevalence of
the different subpopulations also depends on these covariates (expressed through nonzero
values of $\balpha_1,\ldots,\balpha_K$).

When making predictions, it is valuable to be able to quantify one's
level of uncertainty.  In this paper, we develop the machinery
necessary to do so in the context of the high-dimensional mixture of
the experts model described above.  In particular, given a sample of $n$
observations from the model in
\eqref{eq:normal}--\eqref{eq:multinomial}, a confidence level $q\in(0,1)$, and a new predictor vector of
interest $\xnew$, we show how to form a
properly calibrated {\em prediction set} $\Omega_q(\xnew)$.  That is,
given a new draw $(\xnew,\ynew)$ from \eqref{eq:normal}--\eqref{eq:multinomial}, we have that
\begin{align}
\P(\ynew\in \Omega_q(\xnew)|\xnew)\ge 1-q. \label{eq:predset}
\end{align}
\begin{figure}
	\centering
	\scalebox{1}{\input{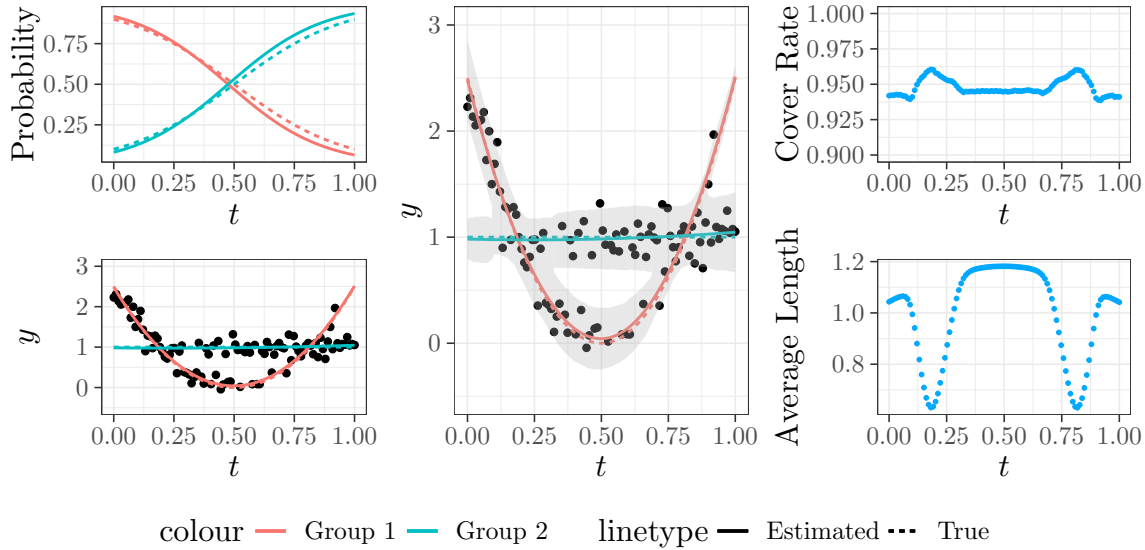}}
	\caption{(Top left) True (dashed) and estimated (solid) class probabilities
          $\pi_k(t)$ and $\hat\pi_k(t)$. (Bottom left) Data set of
          $n=100$ points  $\{(t,y_t)$ (black dots) with true (dashed)
          and estimated (solid) mean functions $\mu_k(t)$ and
          $\hat\mu_k(t)$. (Middle) Prediction set at $95\%$ confidence
          level $\Omega_{0.05}(t)$ for $t\in[0,1]$ (shaded
          area). (Right) Monte Carlo estimate of coverage and length
          of prediction sets $\Omega_{0.05}(t)$ across $500$ training
          sets and $1000$ $\ynew$ per training set.}
	\label{fig:intro_example}
\end{figure}

As an illustration, consider the toy example shown in Figure
\ref{fig:intro_example}.  There are $K=2$ subpopulations, one
following a quadratic relation and the other being constant (for
details of the construction, see Section \ref{sec:empirical}).  The
upper panel shows that the mixture weight on the quadratic
subpopulation decreases with increasing $t$.  This can also be seen by
inspecting the scatterplot and noting that for large $t$, most of the points are in
the constant subpopulation.  The solid lines show our estimated mixture
of experts model fit based on the points in the scatterplot.  Now
suppose we are about to observe a new point $\ynew$ at, say, $t_{\text{new}}=0.5$.
Can we form a set $\Omega_{0.95}(t_{\text{new}})$ that is guaranteed
to capture $\ynew$ at least 95\% of the time? Given that we do not
know which subpopulation $\ynew$ will be drawn from,
$\Omega_{0.95}(t_{\text{new}})$ will be a union of two intervals.  The
size and location of these intervals will depend on our estimates of
the subpopulation means and variances (governed by $t_{\text{new}}$ and our
estimates of $\bbeta_1$, $\bbeta_2$, $\sigma_1$, and $\sigma_2$) as
well as the estimated mixture weights (governed by $t_{\text{new}}$
and our estimates of $\balpha_1$ and $\balpha_2$).  The gray band in
the middle panel shows our constructed $\Omega_{0.95}(t_{\text{new}})$ as we vary
$t_{\text{new}}$ from 0 to 1.  When the means of the populations are
far apart, the prediction set is a union of two intervals, while the
set becomes a single interval when the means are close to each other.
The larger a subpopulation's mixture weight, the wider that interval
becomes.  This ``strategy'' is reasonable because if it's unlikely
that a point will fall in a certain subpopulation, one can get away
with less good coverage of that subpopulation (and thereby reduce the
overall size of the prediction set).  The rightmost panel shows the
result of a simulation in which we generated 500 training sets, each
time constructing $\Omega_{0.95}(t)$ as
a function of $t$; then, we generated 1000 $\ynew$ at each $t$ and
computed the coverage rate (averaging over the $500\cdot 1000$
repetitions for each $t$). This verifies that our procedure approximately attains the
nominal 95\% coverage.  The bottom panel shows the average size of the
prediction set.  Quite intuitively, the prediction set is smallest when the two subpopulations are
very close to each other.  One observes a bit of overcoverage when the
subpopulations are close, which makes sense since in this situation the
interval from one subpopulation can sometimes cover a point from the opposite subpopulation.

Constructing prediction sets in the context of a high-dimensional
mixture of experts model is challenging.  In fact, making any sort of precise statements about even the most
simple mixture of regression models is nontrivial.  For example, much
effort has gone into understanding the convergence and estimation
error of the expectation-maximization estimator \citep{dempster1977maximum}
used in fitting such models (\citealt{yi2014alternating,
  balakrishnan2017, klusowski2019, kwon2019, kwon2020}) as well as
being able to test whether there are two groups versus one \citep{zhu2004hypothesis}.  Adding
high dimensionality to the study of mixture of regression models
brings additional challenges.  \citet{Stadler2010} and
\citet{yi2015regularized} proposed different
$\ell_1$-regularized maximum likelihood estimators with accompanying
estimation error results.  \citet{wang2014high} take this a step further
and develop a truncation-based high-dimensional estimator with both
estimation error results and the ability to construct confidence
intervals for low-dimensional components of the parameter vector.  While their results hold for general
latent variable models, their application to mixture of regression
models is more of a proof of concept, with $K=2$, $\sigma_1=\sigma_2$
assumed known, $\pi_1=\pi_2=0.5$, and $\bbeta_1=-\bbeta_2$.
\citet{zhang2020} provide inference for individual coefficients and
differences of the form $\bbeta_{1j}-\bbeta_{2j}$
within the context of this model using an $\ell_1$ penalty.  They
generalize to the case of unknown mixture weight $\pi_1$, $\bbeta_1\neq-\bbeta_2$, and an
unknown covariance matrix for $X$ (which they take to be multivariate
normal), but they still assume $K=2$ and that the value
$\sigma_1=\sigma_2$ is known. While we are able to adopt in part a similar debiased approach,
we will highlight later why their technique, which works in the $K=2$
mixture of regression setting (which involves a single unknown mixture
parameter $\pi_1$) does not easily generalize to our setting of a
mixture of experts model, in which mixture weights depend on the
unknown parameter vectors $\balpha_1,\ldots, \balpha_K\in\reals^p$.

Furthermore, our interest in forming a prediction set \eqref{eq:predset} requires
the ability to make inferential statements about $\xnew^T\bbeta_k$,
not just low-dimensional components of $\bbeta_k$.  In this sense,
\citet{cai2021optimal} pursue a similar goal in performing inference for individualized
treatment effects $\xnew^T(\bbeta_1 - \bbeta_2)$; however, unlike the
mixture of regression setting, the group memberships are known in
their context.

To summarize, to the best of our knowledge, predictive inference in mixture of
expert models has not been addressed in the literature.  Furthermore,
we address this problem for general $K$ and in the high-dimensional
setting.  We  make use of ideas from the debiased lasso literature.
The debiasing approach for constructing confidence intervals for coefficients has been
widely used in  linear regression models in high-dimensional settings
(\citealt{javanmard2014confidence, javanmard2014hypothesis,
  zhang2014confidence, van2014asymptotically,
  javanmard2018debiasing}). In recent years, there has been work
on inference for general linear functions (\citealt{cai2017confidence,
  guo2021group, javanmard2020flexible, zhu2018linear}). In terms of
debiasing in the non-mixture setting, \citet{cai2017confidence,
  tripuraneni2019debiasing, athey2018approximate} proposed
bias-corrected estimators for a single linear regression model while, as
we have noted above, \citet{cai2021optimal} considers inference for
$\xnew^T (\bbeta_1 - \bbeta_2)$ in the case of two observed (i.e., non-latent) groups.

Interest in predictive inference has led to an active area of work on
conformal prediction.  These approaches are attractive for being
distribution-free and providing finite-sample coverage (see, e.g., \citealt{papadopoulos2002inductive,vovk2005,lei2018,romano2019}).  They rely on
very general ideas such as the exchangeability of draws from the
distribution.  However, the coverage that these conformal methods
attain is not conditional on $\xnew$ as in \eqref{eq:predset} but
rather holds marginally over $\xnew$:

\begin{align*}
\P(\ynew\in \Omega_q(\xnew))\ge 1-q.
\end{align*}

Indeed, it has been proven that to obtain finite length sets with conditional coverage,
one needs to make stronger assumptions
\citep{vovk2012conditional,lei2014distribution}.  In certain
applications, one specifically desires coverage of the form
\eqref{eq:predset} and assuming the parametric form
\eqref{eq:normal}--\eqref{eq:multinomial} can be a small price to
pay.  For example, in the oceanographic example of mixture of experts
\citep{hyun2020}, the mixture of Gaussians structure is visually
well-supported.  Prediction sets with conditional coverage are
desirable because we would like to be able to say that for a given set of
environmental conditions (e.g., at a specific temperature and salinity
level) our prediction set for the phytoplankton diameters will have a
95\% coverage guarantee.  Marginal coverage would mean that if we make
predictions over many randomly sampled environments, our coverage
would average out to 95\%.  The latter means, for example, that a
procedure could be overconfident (i.e., undercovering) at high
temperatures and underconfident (i.e., overcovering) at low
temperatures.  Setting conditional coverage \eqref{eq:predset} as the
goal guards against this undesirable property.

The rest of the paper is organized as follows.  In Section \ref{sec:methodology},
we describe our approach to constructing prediction sets.  This
involves estimating parameters using a penalized
expectation-maximization approach (Section \ref{sec:method}), then using
a debiasing technique on the coefficient vectors from each of the
component distributions (Section \ref{sec:debias}), and finally combining
$K$ intervals into a prediction set in a fashion that maintains proper
coverage (Section \ref{sec:prediction_sets}).  In Section \ref{sec:theory}, we provide
theoretical guarantees that establish the asymptotic validity of our
constructed prediction sets and provide insight into the conditions
under which we expect nominal coverage to hold.  In Section \ref{sec:empirical}
we investigate the empirical performance of our prediction sets in a
variety of settings.  Section \ref{sec:data} shows our sets and evaluates
their performance empirically in predicting the
critical temperatures of superconducting materials. We conclude this
section with some notation.

\bigskip

\noindent\textbf{Notation.} Throughout the paper, we use $[p]$ for the set of integers $1, ... , p$. We use $\e_i$ to denote the $i$-th standard basis vector. For a vector $\x \in \mathbb{R}^p$, we denote $\qnorm{\x} = \left(\sum_{j=1}^p |x_j|^q\right)^{1/q}$ for $q > 0$ and $\|\x\|_0 = |\supp(\x)|$, $\infnorm{\x} = \max_{j\in [p]} |x_j|$. For a matrix $\boldsymbol{A} \in \mathbb{R}^{n \times n}$, we use $\infMnorm{\boldsymbol{A}} = \max_{1 \leq i,j \leq n}|A_{i,j}|$.  For sequences $a_n$ and $b_n$, 
we use the notation $a_n \asymp b_n$ to indicate that $a_n$ is bounded both above and below by $b_n$ asymptotically, i,e, for some constants $C_0, C_1$ and for all $n\ge n_{n_0}$ we have $C_0\le |a_n/b_n|\le C_1$. In addition, we write $a_n = O_p(b_n)$ if for any $\eps>0$, there exists $C_\eps>0$ and large enough $n_\eps$ such that $\P(|a_n/b_n| >C_\eps)<\eps$, for all $n\ge n_\eps$. 
 we write $a_n = o_p(b_n)$ if $a_n/b_n$ converges to zero in probability, i.e., $\lim_{n \to \infty} \P\left(|a_n/b_n| \geq \eps
 \right) =0, \forall \eps>0$. The notation $\dist$ indicates convergence in distribution.

We use the notation $\|\cdot\|_{\psi_2}, \|\cdot\|_{\psi_1}$ to refer to the sub-Gaussian and sub-exponential norms respectively.
Specifically, for a random variable $X$, we let
\[
\|X\|_{\psi_1} = \sup_{q\ge 1} \;q^{-1} (\E|X|^q)^{1/q}\,, \quad 
\|X\|_{\psi_2} = \sup_{q\ge 1} \;q^{-1/2} (\E|X|^q)^{1/q}\,.
\]
For a random vector $\x$, its sub-Gaussian and sub-exponential norms are defined as
\[
\|\x\|_{\psi_1} = \sup_{\twonorm{\u}\le 1} \;\|\<\x,\u\>\|_{\psi_2} \,, \quad 
\|\x\|_{\psi_1} = \sup_{\twonorm{\u}\le 1} \;\|\<\x,\u\>\|_{\psi_1}\,.
\]

\section{Methodology}\label{sec:methodology}

We begin (in Section \ref{sec:method}) with a review of a penalized
maximum likelihood procedure for the MoE model
\eqref{eq:normal}--\eqref{eq:multinomial} such as is used in \citet{hyun2020}.  Sections \ref{sec:debias}
and \ref{sec:prediction_sets} then introduce our proposed methodology
for forming prediction sets in this context.

\subsection{Penalized EM-based Estimator}\label{sec:method}
The expectation-maximization (EM) algorithm
(\citealt{dempster1977maximum}) is a common heuristic when faced with
maximum-likelihood problems involving missing data, especially in the
form of latent variables. The algorithm operates in an iterative
fashion, alternating between the E (expectation) step and M
(maximization) step, while managing to increase the objective function. 

Assume $n$ data points $\{(\x_i, y_i)\}_{i=1}^n$ are drawn
independently from the MoE model \eqref{eq:normal}--\eqref{eq:multinomial}. The EM algorithm aims at maximizing the log-likelihood
\begin{align}\label{eq:log-like}
\ell(\btheta)  = \frac{1}{n} \sum_{i=1}^n \log\left[\sum_{k=1}^K \pi_{k} (\x_i) \cdot \phi_{k} (\x_i, y_i) \right],
\end{align}
where $\btheta = \left(\{\bbeta_k, \balpha_k, \sigma_k\}_{k\in[K]}
\right)\in \reals^{(2p+1)K}$ represents the model parameters, $\phi_k(\x_i, y_i)$ is given by
\begin{align}\label{eq:phi-def}
\phi_{k}(\x_i, y_i) = \frac{1}{\sqrt{2\pi}\sigma_k} \exp\left( -\frac{(y_i- \x_i^{T}\bbeta_k)^2}{2\sigma_k^2}\right),
\end{align}
and $\pi_k(\x_i)$ is given in \eqref{eq:multinomial}.
The log-likelihood $\ell(\btheta)$ is not concave even as a function
of $\bbeta_k$, treating $\balpha_k$
and $\sigma_k$ as fixed. This is due to the marginalizing over the
latent cluster memberships $z_i$. The EM algorithm instead employs a
minorize-maximize approach \citep{hunter2004tutorial} in which a
minorizer to the log-likelihood is repeatedly constructed and 
maximized. More specifically, given some fixed $\hbtheta$ it maximizes a lower bound function $Q(\btheta|\hbtheta)$ over $\btheta$ to make $\ell(\btheta) - \ell(\hbtheta)$ large. We refer to~\cite{hastie2009elements} for a more detailed introduction to EM algorithm and derivation of the function $Q(\btheta|\hbtheta)$ and here only provide the description of the EM algorithm for the MoE model.

Let $\gamma_{i,k}(\btheta)$ be the probability that $z_i = k$ conditioned on the observed variable $(\x_i,y_i)$, i.e.,
\[
\gamma_{i,k}(\btheta):= \P\left(z_i = k | \x_i, y_i  \right) = \frac{\pi_{k}(\x_i) \cdot \phi_{k}(\x_i, y_i)}{ \sum_{\ell = 1}^K \pi_{\ell}(\x_i) \cdot \phi_{\ell}(\x_i, y_i)}.
\]
The function $\gamma_{i,k}(\btheta)$ is sometimes referred to as
\emph{responsibilities} in the literature (see, e.g. \cite{hyun2020})
because it quantifies how ``responsible'' group $k$ is for point $i$.

Consider the function $Q(\btheta|\hbtheta)$ defined as
\begin{align}\label{eq:Q}
Q(\btheta|\hbtheta)  =  -\frac{1}{n}\sum_{i=1}^n \sum_{k=1}^K \gamma_{i,k}(\hbtheta) \left[\log \pi_{k}(\x_i)  + \log\phi_{k}(\x_i, y_i) \right] + \sum_{k=1}^K \lambda_{\balpha}\|\balpha_k\|_1 + \sum_{k=1}^K \lambda_{\bbeta}\|\bbeta_k\|_1,
\end{align}
where the regularization terms $\onenorm{\bbeta_k},
\onenorm{\balpha_k}$ are added to enforce sparsity on the estimated
parameters and allow for applications in the high-dimensional sparse regime.

The EM algorithm iterates between estimating the conditional membership probabilities $\gamma_{i,k}(\hbtheta)$, and reducing \eqref{eq:Q} by updating the parameters $\hbtheta$. The details of the updates are given below: 
\begin{itemize}
	
	\item E step: Estimating the conditional responsibility of membership based on the latest parameter estimate.
	\[
	\gamma_{i,k}(\hbtheta) = \frac{\frac{\hat{\pi}_{i,k}}{\hsigma_k} \exp\Big( -\frac{(y_i- \x_i^{T}\hbbeta_k)^2}{2\hsigma_k^2}\Big)}{\sum_{\ell = 1}^K \frac{\hat{\pi}_{i,\ell}}{\hsigma_{\ell}} \exp\Big( -\frac{(y_i- \x_i^{T}\hbbeta_{\ell})^2}{2\hsigma_{\ell}^2}\Big)}.
	\]
	
	\item M step: Updating $\hbtheta$ to lower the objective value in  \eqref{eq:Q}.
	\begin{itemize}
		
		\item For each $k=1,\ldots, K$, update $\hbbeta_{k}$:
		\begin{align}\label{eq:beta}
		\hbbeta_{k}= \underset{\bbeta}{\arg \min} \sum_{i=1}^{n} \frac{\gamma_{i,k}(\hbtheta)}{\hsigma^2_k} \frac{(y_i - \x_i^{T}\bbeta)^2}{2n} + \lambda_{\bbeta} \|\bbeta\|_1.
		\end{align}
		
		\item Update $\{\hbalpha_{k}\}_{k=1}^K$:
		\begin{align}\label{eq:alpha}
		\{\hbalpha_{k}\}= \underset{\{\balpha_k\}}{\arg \min} -\frac{1}{n}\sum_{i=1}^{n}\Big( \sum_{k=1}^K \gamma_{i,k}(\hbtheta) \x_i^T \balpha_k - \log\Big(\sum_{\ell=1}^K \exp{(\x_i^T \balpha_{\ell})}\Big)\Big) + \lambda_{\balpha} \sum_{k=1}^K \|\balpha_k\|_1.
		\end{align}

		\item For each $k=1,\ldots, K$, update $\hsigma_k$:	
		\begin{align}
		\hsigma^2_{k} = \frac{\sum_{i=1}^{n} \gamma_{i,k}(\hbtheta) (y_i - \x_i^{T} \hbbeta_k )^2}{\sum_{i=1}^n \gamma_{i,k}(\hbtheta)}.
		\end{align}
	\end{itemize}
\end{itemize}

As noted in \citet{hyun2020}, this M-step update represents a decreasing in the objective but not a
minimization since $(\bbeta_k,\sigma_k)$ would need to be jointly optimized.  The algorithm terminates when the improvement in \eqref{eq:Q} is below some threshold.

In practice, we use the R package \code{flowmix} (\citealt{flowmix}) to carry out the EM algorithm. 

\subsection{Debiased Prediction}\label{sec:debias}

From \eqref{eq:normal}--\eqref{eq:multinomial}, we
have that
$$
\ynew|\xnew\sim N(\xnew^T\bbeta_k,\sigma_k^2) \text{ with probability } \pi_k(\xnew).
$$
Because of the penalization, $\hat{\Gamma}_k:= \xnew^T \hbbeta_k$ is a
biased estimate of $\Gamma_k:= \xnew^T \bbeta_k$, and thus would give
biased predictions for $\ynew$ even if we knew that $\znew=k$.  We
therefore propose a debaising procedure.  Given an arbitrary vector $\u_k$, we construct a debiased prediction
\begin{align}\label{def:debias}
\hat{\Gamma}^d_k := \xnew^T \hbbeta_k + \u_k^T \frac{1}{n}\sum_i \frac{\partial \ell_i (\hbtheta)}{\partial \bbeta_k},
\end{align}
where $\partial/\partial \bbeta_k$ denotes the gradient with respect
to $\bbeta_k$.

Our proposed choice for $\u_k$ depends on the estimated sample Fisher
information matrix, which is given by 
\begin{align}\label{eq:I}
\tilde{\boldsymbol{I}}(\hbtheta) = \frac{1}{n} \sum_{i=1}^n \nabla \ell_i(\hbtheta) \nabla \ell_i^T(\hbtheta),
\end{align}
where $\ell_i(\hbtheta)$ are the summands in \eqref{eq:log-like} evaluated at $\hbtheta = \left(\{\hbbeta_k, \hbalpha_k, \hsigma_k\}_{k\in[K]} \right)$, and
\[\nabla \ell_i = \Big(\Big\{\frac{\partial \ell_i}{\partial \bbeta_k}\Big\}_{k=1}^K, \Big\{\frac{\partial \ell_i}{\partial \balpha_k}\Big\}_{k=1}^K, \Big\{\frac{\partial \ell_i}{\partial \sigma_k}\Big\}_{k=1}^K\Big)^T\,.\]
We would like a choice of $\u_k$ that will lead to a narrower interval.
This intuition (which is described in greater detail below) suggests
choosing $\u_k$ as the solution to the optimization problem
\begin{subequations}\label{eq:debiase_opt}
\begin{align*}
&\min_{\u} \ \u^T \SMI  \u   \qquad \qquad \\
&\ \textrm{s.t.} \ \ \sup_{\omega \in \mathcal{C}} \left| \left\langle \omega, \tilde{\bSigma}_k \u - \xnew\right\rangle \right| \leq \lambda_k \|\xnew\|_2 , \quad 
\onenorm{\u}\le L\twonorm{\xnew},\tag{\ref{eq:debiase_opt}}
\end{align*}
\end{subequations}
where we define $\tilde{\bSigma}_k := \frac{1}{n} \sum_{i=1}^n  \frac{\gamma_{i,k}(\hbtheta)}{\hsigma_k^2}\x_i \x^T_i$, 
$L$ is a sufficiently large constant, and $\lambda_k$ is a tuning parameter (in Section~\ref{sec:theory}, we discuss the proper rate of $\lambda_k$ and the choice of $L$). In addition, $\mathcal{C} = \{ \e_1, ..., \e_p, \xnew/\|\xnew\|_2\}$, where $\e_i$ denotes the $i$-th standard Euclidean basis vector.  The matrix $\SMI$ is the estimated Fisher information matrix, constrained to $\bbeta_k$, which is given by
\begin{align}\label{eq:sample-Fisher}
\SMI = \frac{1}{n} \sum_{i=1}^n \left( \gamma_{i,k}(\hbtheta)\frac{y_i - \x_i^T \hbbeta_k}{\hsigma_k^2} \x_i\right) \left( \gamma_{i,k}(\hbtheta)\frac{y_i - \x_i^T \hbbeta_k}{\hsigma_k^2} \x_i\right)^T.
\end{align}
The above characterization follows from the definition of $\tilde{I}(\hbtheta)$ given by~\eqref{eq:I} along with identity~\eqref{eq:Q-grad}, restricted to class $k$.

\begin{remark}
  \label{rem:split}
 Because $\tilde{\bSigma}_k$ and $\SMI$ depend on both the response
 vector ${\bf y}$ and the covariate matrix $\mathbf{X}$, the solution $\u_k$ to \eqref{eq:debiase_opt} is also dependent
on them. This is in contrast to the optimization problem proposed by
\citet{javanmard2014confidence} for constructing the debiasing
direction, as it only involved covariates and hence conditional on the
covariate matrix $\mathbf{X}$, the debiasing direction was independent
of the response vector ${\bf y}$. 

To deal with the complications resulting from the dependence of $\u_k$
on $(\mathbf{X},\mathbf{y})$, we do sample splitting. Specifically, we
split the data into two sets, $\mathcal{D}_1$ and $\mathcal{D}_2$,
where optimization \eqref{eq:debiase_opt} is solved on
$\mathcal{D}_1$, and $\hbbeta_k,\hat{\Gamma}^d_k$ are calculated using
samples in $\mathcal{D}_2$, per \eqref{def:debias}.
\end{remark}

The above approach to construct the direction $\u_k$ is generalized
from \citet{javanmard2014confidence, zhang2014confidence,
  cai2017confidence} and aims to find a direction that minimizes the
variance while controlling the bias. However, the goal in
\citet{javanmard2014confidence, zhang2014confidence,
  cai2017confidence} is to establish inference for coefficients, which
differs from our goal of establishing inference for prediction. Closer
to our aim here, we follow the proposal of \citet{cai2021optimal} to construct
prediction intervals, under a linear regression model; however, our
setting of a MoE model is more complex and requires novel methodology and analysis.  In particular, our optimization problem differs from \citet{cai2021optimal} in that we use a constrained Fisher information matrix, $\SMI$, in the objective of \eqref{eq:debiase_opt}, and we allow the matrix $\tilde{\bSigma}_k$ in the constraint of \eqref{eq:debiase_opt} to be different from the matrix in the objective, while \citet{cai2021optimal} keeps them the same.

The first constraint in~\eqref{eq:debiase_opt} can be decomposed as
\begin{align}
\infnorm{\tilde{\bSigma}_k \u - \xnew} &\leq \lambda_k \twonorm{\xnew} \label{eq:debiase_opt1}\\
\left| \left\langle \xnew, \tilde{\bSigma}_k \u - \xnew\right\rangle \right| &\leq \lambda_k \twonorm{\xnew}^2 \label{eq:debiase_opt2}
\end{align}
Similar to the general intuition behind the quadratic optimization of
\citet{javanmard2014confidence}, the objective value $\u^T \SMI  \u$
is related to the variance of $\hat{\Gamma}^d_k$, and the constraint
\eqref{eq:debiase_opt1} relates to its bias. So the optimization is
indeed aiming to minimize the variance (and hence the length of
prediction intervals which will be constructed based on
$\hat{\Gamma}^d_k$), while controlling the bias of
$\hat{\Gamma}^d_k$. That said, in the analysis we need to show that
the bias is dominated by the variance term and therefore need to
establish a lower bound on the variance. The
constraint~\eqref{eq:debiase_opt2} is added for this step. It makes
the feasible set of the optimization problem smaller and makes it
possible to lower bound the optimal value of the objective (see Proposition~\ref{lemma:fisher_RE} for technical arguments). This idea originates from \citet{cai2021optimal}, which introduced
the ``variance-enhancement projection direction''. While the general intuition carries over to our current setting, characterizing the statistical properties of $\hat{\Gamma}^d_k$ under the MoE requires a rather intricate and technical analysis.

We denote the estimated variance for $\hat{\Gamma}^d_k$ as
\begin{align}
\widehat{V}_k = \frac{1}{n}\u^T_k \SMI \u_k,  \quad k=1, ..., K,
\end{align}
and the prediction variance estimate (conditional on $\znew=k$) as 
\begin{align}\label{eq:b_k}
b^2_k: = \widehat{V}_k + \hsigma^2_k, \quad k=1, ..., K.
\end{align}

\subsection{Prediction Sets}\label{sec:prediction_sets}

Recall that our goal is to construct a $100(1 - q) \%$
prediction set $\Omega_q(\xnew)$ satisfying \eqref{eq:predset}. By the
nature of the mixture, we seek a prediction set of the form 
\[
\Omega_q(\xnew) = \cup_{k=1}^K \mathcal{L}_k,
\]
where each $\mathcal{L}_k = [l_k, u_k]$ is centered at a  debiased
estimator $\hat{\Gamma}^d_k$. Regarding the length of
$\Omega_q(\xnew)$ as our budget it is clear at an intuitive level that
we should spend more on the mixture components to which
$(\xnew,\ynew)$ is more likely to belong, i.e., to those groups $k$ with larger $\pi_k(\xnew)$.
To this end, we form a probability density function using a weighted mixture of Gaussian densities:
\begin{align}\label{eq:mixed_gaussian}
f(y) = \sum_{k=1}^K \frac{\hat{\pi}_k(\xnew)}{b_k} \cdot \phi\left( \frac{ y-\hat{\Gamma}^d_k }{b_k}\right),
\end{align}
where $\phi(\cdot)$ is the standard normal pdf. The particular form of this density is justified in the proof of Theorem \ref{thm:coverage} given in Section \ref{proof:thm:coverage}. We give a schematic illustration of $f$ in Figure~\ref{fig:illustration}. 

\begin{figure}[t]
\centering
\includegraphics[width=6cm]{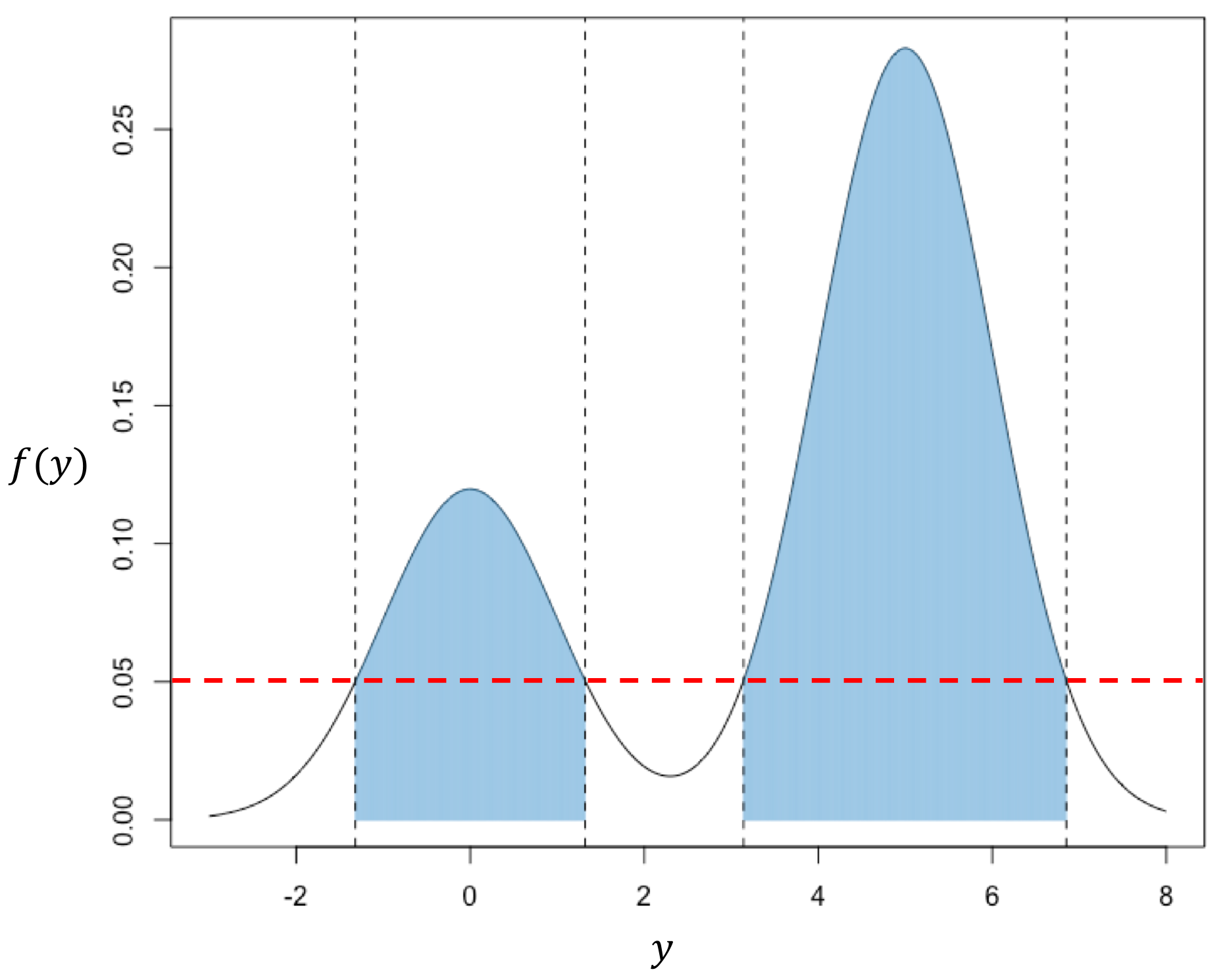}
\caption{An illustration of the mixture density $f$, given
  by~\eqref{eq:mixed_gaussian} for $K=2$ groups. The cutoff level
  (indicated by the red line) is the highest level such that the region
    shaded in blue has area at least $1-q$.}
    \label{fig:illustration}
\end{figure}

We seek a set of intervals $[y^{-}_1, y^{+}_1], ..., [y^{-}_K, y^{+}_K]$, such that 
\begin{align}
\sum_{k=1}^K \int_{y^{-}_k}^{y^{+}_k} f(y) dy = 1-q,
\end{align}
while minimizing $\sum_{k=1}^K |y^{+}_k - y^{-}_k |$. For this, we
start from a large cutoff (the horizontal line in the figure) and decrease that until the area under $f$ and
corresponding to $y$ with $f(y)$ above the cutoff (the blue region in
the figure)
% and above the cutoff
is $1-q$. To approximate the area we take a discretization approach as outlined below. 

Assume without loss of generality that $\hat{\Gamma}_1^d \leq ... \leq \hat{\Gamma}_K^d$. We start by considering the interval 
\[
\mathcal{Q} = [\hat{\Gamma}_1^d -  b_1 z_{1-q/2}  , \quad \hat{\Gamma}_K^d +  b_K z_{1-q/2} ],
\]
which we know has probability at least $1-q$.  We then divide
$\mathcal{Q}$ into segments of size $\delta$, denoted as
$\mathcal{Q}_1, ..., \mathcal{Q}_{\frac{|\mathcal{Q}|}{\delta}}$. The
area under the curve $f(y)$ confined to the segment $\mathcal{Q}_i$ is
approximately $\delta h_i$, where $h_i$ is the density $f$ evaluated at the midpoint point of $\mathcal{Q}_i$. We next sort $h_i$'s corresponding to each segment in decreasing order, i.e., 
\[
h_{(1)} \geq h_{(2)} \geq ... \geq h_{(\frac{|\mathcal{Q}|}{\delta})},
\]
and find the smallest $N$ such that 
\begin{align}\label{eq:sum_hi}
\delta\sum_{i=1}^N h_{(i)} \geq {1-q}.
\end{align}
We return $\Omega_q(\xnew): =\cup_{i=1}^N \ \mathcal{Q}_{(i)}$ as the prediction set.

\begin{algorithm}

	\begin{algorithmic}[1]
		\caption{Constructing prediction set $\Omega_q(\xnew)$
                  with $(1-q)$ coverage}\label{algo}
		\REQUIRE Confidence level $1-q$, discretization scale $\delta$, debiased estimate of each center $\hat{\Gamma}^d_k$, prediction standard error estimate $b_k$, membership probability estimate $\hat{\pi}_k(\xnew), k=1, ..., K$.
		\ENSURE Prediction set with $(1-q)$ coverage
		
		\STATE Form a weighted mixture of Gaussian densities
		\begin{align*}
		f(y) = \sum_{k=1}^K \frac{\hat{\pi}_k (\xnew)}{b_k} \cdot \phi\left( \frac{ y-\hat{\Gamma}^d_k }{b_k}\right).
		\end{align*}
		\STATE Choose a large enough interval $\mathcal{Q}$ so that the integral $\int_{\mathcal{Q}} f(y)\geq 1-q$, i.e.
		\[\mathcal{Q} = [\hat{\Gamma}_1^d -  b_1 z_{1-q/2}, \quad \hat{\Gamma}_K^d +  b_K z_{1-q/2} ]\,.\]
		\STATE Divide $\mathcal{Q}$ into segments of size $\delta$. Let $y_i$ be the midpoint of $\mathcal{Q}_i$ and $h_i = f(y_i)$.
		\STATE Sort $h_i$'s in decreasing order,
		\[
		h_{(1)} \geq h_{(2)} \geq ... \geq h_{(\frac{|\mathcal{Q}|}{\delta})}.
		\]
		\STATE Find the smallest $N$ such that 
		\begin{align*}
		\delta \sum_{i=1}^N  h_{(i)} \geq {1-q}.
		\end{align*}
		\STATE Return the union of the corresponding segments 
		\begin{align}\label{eq:P-set}
		\Omega_q(\xnew) = \cup_{i=1}^N \ \mathcal{Q}_{(i)}\,.
		\end{align}
	\end{algorithmic}
	
\end{algorithm}

\section{Theoretical Guarantees}\label{sec:theory}

We a consider sequence of problems where the sample size $n\to\infty$
and covariate dimension $p = p(n)\to\infty$, while the number of
groups $K$ is bounded, and we establish asymptotic validity of our prediction sets for the MoE model \eqref{eq:normal}--\eqref{eq:multinomial}. 
We first lay out several technical assumptions on the estimation error
of $\hbtheta$, the random covariate vectors $\{\x_i\}_{i\in [n]}$, and
the model parameters $\btheta$.  

\begin{itemize}
	\item (A1) \emph{Parameter estimation $\est$.} Suppose that
	\[
	\max_{k\in[K]}\onenorm{\hbtheta_k - \btheta_k}  = \max_{k\in[K]}\left(\onenorm{\hbbeta_k - \bbeta_k}+
	\onenorm{\hbalpha_k - \balpha_k} + |\hsigma_k - \sigma_k|\right) = O_p(\est)\,,
	\]
	where $\est$ scales with $n, p$ and potentially other
        structure associated with the parameters (e.g. sparsity levels). We assume that   
	\begin{align}\label{eq:main-condition}
	\sqrt{n} \log(np)  \est^2 = o(1)\,.
	\end{align}
	\item (A2) \emph{Distribution of features.} We have a
          positive-semidefinite matrix $\bSigma\in\reals^{p\times p}$
          (or more precisely a sequence of matrices of growing dimension) with bounded operator norm, $\|\bSigma\|_{{\rm op}}\le C_{\Sigma}$, for a constant $C_{\Sigma}$ as $p\to \infty$. Suppose that $\bSigma^{-1/2}\x_i$ are independent sub-Gaussian vectors, with mean zero and sub-Gaussian norm $\|\bSigma^{-1/2}\x_i\|_{\psi_2} = O(1)$.
	\item (A3) \emph{Bounded noise and signal.} The noise variances $\sigma_k^2$ are strictly positive and bounded constants for $k\in[K]$. We also assume that $\max_{\ell, k\in[K]}\|\bbeta_k - \bbeta_{\ell}\|_2 = O(1)$.
\end{itemize}
Condition (A1) assumes an $\ell_1$-consistency rate for the estimate
$\hbtheta$. Consistency presupposes identifiability, which in the case
of the $\{\balpha_k\}$ may require additional assumptions (since
$\{\balpha_{k}+{\bf c}:k\in [K]\}$ corresponds to the same $\pi_k(\x)$
as $\{\balpha_{k}:k\in [K]\}$). Our theoretical results apply to any estimator which
satisfy condition \eqref{eq:main-condition}. The proposed EM estimator
in Section~\ref{sec:method} is just one specific choice. Instead of
$\ell_1$-regularization, one can follow other variants based on
iterative truncation. For example,~\cite{wang2014high} analyzes a
mixture of regression model with two groups and proposes a truncated
EM algorithm (with a gradient ascent implementation) which achieves $\est  = s\sqrt{\log(p)\log(n)/n}$, with $s$ the sparsity level of model parameters. The work~\cite{zhang2020} derives a similar $\ell_1$-consistency rate for the high-dimensional mixed linear regression with two groups, for an iterative EM procedure which performs $\ell_1$ regularization at each step.  
While the mixture of experts model is more complicated we conjecture that a similar rate for $\est$ carries over to this setting. Under such conjecture, condition \eqref{eq:main-condition} simplifies to 
\[
\frac{s^2 \log(np) \log(p)\log(n)}{\sqrt{n}} = o(1)\,.
\]

Condition (A2) is on the random covariate vectors $\x_i$ and is a common assumption in high-dimensional statistical estimation; see e.g.~\cite{buhlmann2011statistics}. Condition (A3) on the pairwise distances $\twonorm{\bbeta_k - \bbeta_\ell}$ and noise variance $\sigma_k^2$ is to control the heterogeneity of data coming from different groups.

Our first theorem is on asymptotic normality of the bias-corrected
estimators $\hat{\Gamma}^d_k$ defined in \eqref{def:debias} and involves the matrix
$$
\bSigma_{k}=
\mathbb{E}\left[\frac{\gamma_{1k}\left(\btheta\right)}{\sigma_{k}^{2}}\x_{1}\x_{1}^{T}\right].
$$
\begin{thm}\label{thm:normality}
Suppose that $\frac{\onenorm{\bSigma^{-1}_k\xnew}}{\twonorm{\xnew}} =
O(1)$ and $\log(p) =o(n^{1/4}/\sqrt{\log(n)})$. Let $\u_k$ be the solution to the optimization problem
\eqref{eq:debiase_opt} with constant $L\ge \frac{\onenorm{\bSigma^{-1}_k\xnew}}{\twonorm{\xnew}}$ and $\lambda_{k}\asymp \est \log(np) + \sqrt{\log(p)/n}$.
	Under Assumptions (A1), (A2), (A3), and Remark
        \ref{rem:split}, we have
	\begin{align}\label{eq:normality}
	\frac{\sqrt{n}(\hat{\Gamma}_k^d - \Gamma_k)}{ \sqrt{\u^T_k \SMI \u_k}} \dist N( 0, 1),
	\end{align}
	where $\SMI$, given by~\eqref{eq:sample-Fisher}, is the sample Fisher information matrix constrained to entries corresponding to $\bbeta_k$.
\end{thm}

Now that we have established the asymptotic normality of our debiased estimators, we are ready to prove that our prediction sets provide proper asymptotic coverage.

\begin{thm}\label{thm:coverage} 
	Under the assumptions of Theorem~\ref{thm:normality}, the prediction set $\Omega_q(\xnew)$ has asymptotically valid coverage. Specifically, fix $\gamma>0$ arbitrarily small and in Algorithm~\ref{algo} set the discretization scale $\delta$ so that
	\[
	\delta\le 11 \sqrt{\gamma \min_{k\in[K]} (b_k)/{\rm Len}(\mathcal{Q})}\,,
	\]
	with ${\rm Len}(\mathcal{Q})$ representing the length of interval $\mathcal{Q}$.
	Then, for any $\xnew$ and its response $\ynew$ generated according to MoE, we have
	\[
	\lim_{n\to\infty} \P\left(\ynew \in \Omega_q(\xnew)|\xnew\right) \ge 1- q -\gamma\,.
      \]
\end{thm}
Note that choosing $\gamma$ arbitrarily small we get a coverage
arbitrarily close to $1-q$.
We refer to Section~\ref{sec:thm-proof} for the proof of Theorems~\ref{thm:normality} and \ref{thm:coverage}.

\section{Numerical Study}\label{sec:empirical}

In Section \ref{sec:low-dim}, we return to the low-dimensional example given in
Section \ref{sec:intro} and consider several variations to build
greater understanding of the behavior of our intervals.  In Section
\ref{sec:high-dim}, we assess the performance of our procedures in a high-dimensional example.

\subsection{A Low-Dimensional Example}\label{sec:low-dim}

Figure \ref{fig:intro_example} shows a two-group example where the
mean functions of the two groups are
$$
\mu_1(t)=10(t-0.5)^2\qquad\text{and}\qquad\mu_2(t)=1,
$$
the error variances are $\sigma_k^2=0.15^2$,
and the log odds of being in the first group is given by
$$
\log\left[\frac{\pi_1(t)}{1-\pi_1(t)}\right]=\log(9) - 2 \log(9) \cdot t,
$$
for $n=100$ equally spaced values of $t$ ranging from 0 to 1.  This means
that $\pi_1(t)$ decreases from 0.9 to 0.1.
Writing this MoE model in the notation of
\eqref{eq:normal}--\eqref{eq:multinomial}, we have $\x=(1,t,t^2)^T$,
$\bbeta_1=(2.5,-10,10)^T$, $\bbeta_2=(1,0,0)^T$,
$\balpha_1=(\log9,-2\log9,0)^T$, and $\balpha_2=(0,0,0)^T$.
This initial setting is perfectly symmetric in the two groups other
than the difference in mean functions.  For example, the first group at $t=0$ is as
common as the second group is at $t=1$.  This symmetry is manifest in
the width of the intervals: In the middle panel, the width of the interval around the first
group at $t=0$ matches that of the second group at $t=1$.  The
symmetry is also apparent in the third panel plots.

To build our intuition, we explore the effect of breaking this
symmetry between groups.  For example, suppose that the error variances in the two groups are
not equal.  In Figure \ref{fig:example_2}, we take $\sigma_1=0.1$ and $\sigma_2=0.2$.
While the empirical coverage is maintained
around the nominal $0.95$ level, the average length of the prediction
set $\Omega_{0.05}(t)$ is no longer symmetric around $t=0.5$.  This
makes intuitive sense since the class imbalance means that there is more
uncertainty in estimating $\bbeta_2$ than $\bbeta_1$; thus, at
$t=1$, where more of the prediction set is devoted to the second group
(since $\pi_2(1)>\pi_1(1)$), the average length of $\Omega_{0.5}(1)$
will be larger compared to that of $\Omega_{0.5}(0)$.
\begin{figure}
	\centering
	\scalebox{1}{\input{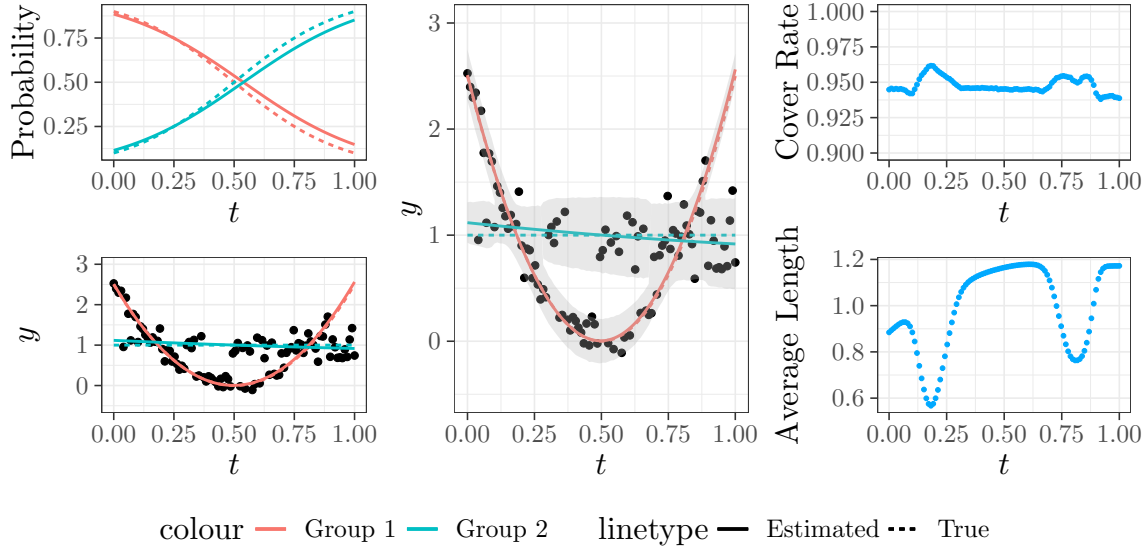}}
	\caption{Unequal error variance case ($\sigma_1<\sigma_2$).
          See caption of Figure \ref{fig:intro_example} for
          description of panels.}
	\label{fig:example_2}
\end{figure}

Another way to break the symmetry would be by considering class
imbalance.  We again assume $\sigma_1=\sigma_2=0.15$, but now suppose
that overall about 60\% of observations belong to the
first group.  In particular, we take instead $\balpha_1=\log10 \cdot
(1,-\log4,0)^T$, so that $\pi_1(t)$ ranges from about $0.91$ at $t=0$
to about $0.29$ at $t=1$. Figure \ref{fig:example_1} shows the effect
of this class imbalance.  We see the same increasing length as in the
previous example despite the error variances being equal.  In this
example, it is the class imbalance that leads to greater uncertainty in estimating $\bbeta_2$.

\begin{figure}
	\centering
	\scalebox{1}{\input{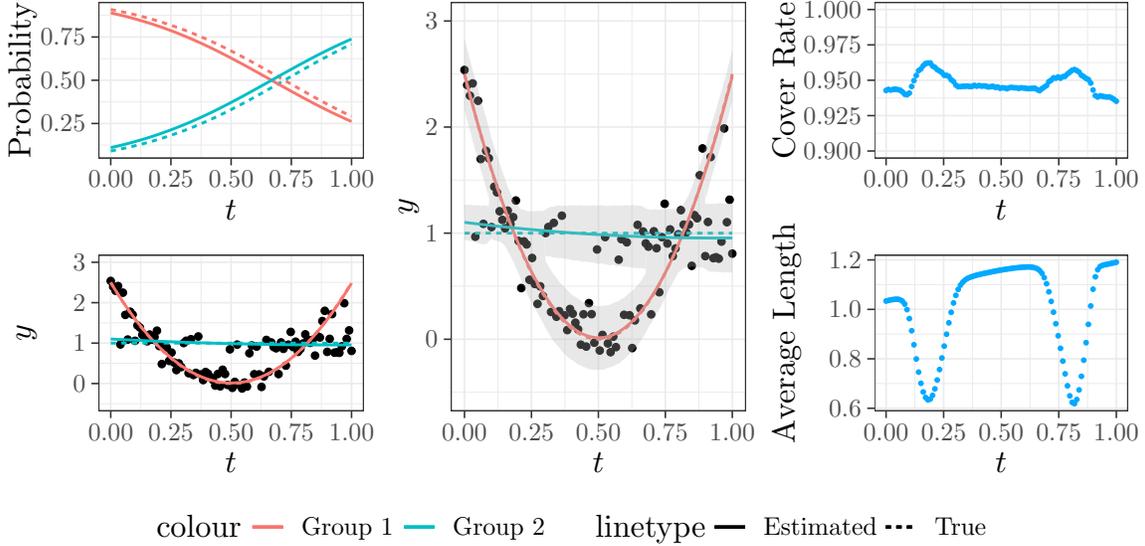}}
	\caption{Class imbalance case ($\E[\pi_1(\x)] > \E[\pi_2(\x)]$).
          See caption of Figure \ref{fig:intro_example} for
          description of panels.}
	\label{fig:example_1}
\end{figure}

\subsection{High-Dimensional Case}\label{sec:high-dim}

We consider a high-dimensional case where $K=2$ and $p=501$ (including
the intercept).  The
oceanographic application in \citet{hyun2020} has repeated
observations at multiple $\x_t$, and we mimic that setup with $T =
150$ feature vectors $\x_t$, and $n_t = 5$ observations per $\x_t$ for
a total of $n=750$ measurements.  This is high-dimensional
since $\btheta$ has dimension $K(2p+1)=2002$.
We take $x_{t1}=1$ and $x_{tj}$ for $j>1$ to be independent standard
Gaussians.  We generate $n_t$ responses $y_{it}$ for each $\x_t$
according to the MoE model \eqref{eq:normal}--\eqref{eq:multinomial}
with $\sigma_1=\sigma_2=1$ and 
\begin{align*}
\bbeta_1 &=& &(&-2&,& 4&,& -2&,& -4&,&  6&,&  2&,& 0&,&0&,&0&,&0&,&0&,&{\bf 0}_{490}^T&)^T,&\\  
\bbeta_{2} &=& &(&2&,& 0&,& 0&,& 0&,& 0&,& 0&,& 4&,& -2&,& -4&,&  6&,&  2&,& {\bf 0}_{490}^T&)^T,&
\end{align*}
and
\begin{align*}
  \balpha_1 &=& &(&  {\bf 0}_{491}&,& 0&,& 0&,& 0&,& 0&,& 0&,& 0.7&,& 0.7&,& 0.7&,& 0.7&,& 0.7&)^T,&\\
\balpha_2 &=& &(&{\bf 0}_{491}&,& -0.7&,& -0.7&,& -0.7&,& -0.7&,& -0.7&,& 0&,& 0&,& 0&,& 0&,& 0&)^T.&
\end{align*}

We estimate the model using the penalized EM algorithm described in
Section \ref{sec:method} as implemented in the \code{flowmix} R
package \citep{flowmix}. We perform five-fold cross validation to
choose the parameters $(\lambda_{\balpha},\lambda_{\bbeta})$ from a
$10 \times 10$ logarithmically-spaced grid.  While in Section
\ref{sec:debias} we remark that sample splitting avoids the
complications resulting from $\u_k$'s dependence on $({\bf X},{\bf
  y})$, empirically we find that coverage is attained even if we
ignore this dependence.  Therefore, in this and all numerical results
we do not use sample splitting for the debiasing step.

To evaluate our method, we generate $100$ independent $\xnew \in
\mathbb{R}^{p}$. For each $\xnew$, we compute
$\Omega_{0.05}(\xnew)$ and record its length. We then generate $100$ independent
$y_{\text{new}, i}$'s for each $\xnew$ and record the proportion of
$y_{\text{new}, i}$'s falling into the prediction set for that $\xnew$.

We repeat the above procedure $500$ times, keeping the $100$ $\xnew$'s the same. For each $\xnew$, we compute

\begin{enumerate}
	\item the average length of prediction sets (across the 500 runs), and 
	\item the coverage probability (proportion across the 500 runs
          and 100 $y_{\text{new}, i}$).
\end{enumerate}

\begin{figure}
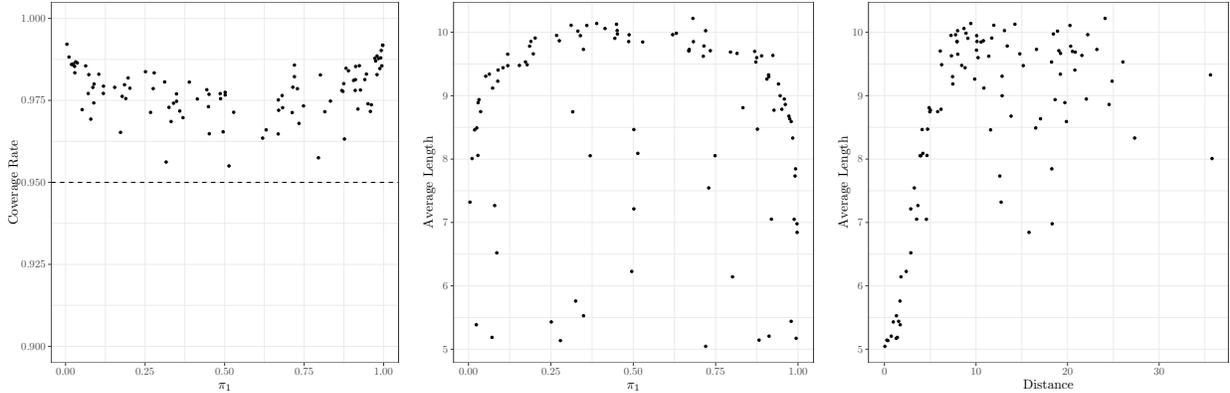

	\centering
	\scalebox{0.42}{\input{figs/k_2_1.tex}}
	\scalebox{0.42}{\input{figs/k_2_2.tex}}
	\scalebox{0.42}{\input{figs/k_2_3.tex}}

	\caption{High-dimensional simulation.  Every point corresponds
          to a different $\xnew\in\reals^p$.  The empirical coverage and average
          length of the prediction set $\Omega_{0.05}(\xnew)$ is computed over 500 training sets and 100
          $y_{\text{new}}$ values for each $\xnew$.}
	\label{fig:high-dim}
      \end{figure}
      
In the left and middle panels of Figure \ref{fig:high-dim}, we plot
these quantities as a function of $\pi_1(\xnew)$, the true probability
of being drawn from cluster 1 for each $\xnew$.  As desired, the
coverage rate of our prediction sets meet the $95\%$ nominal level,
regardless of $\pi_1(\xnew)$. We observe that the
prediction sets tend to be twice as long when $\pi_1(\xnew)\approx0.5$
compared to at the extremes.  This is likely because at an extreme the
less common group's interval can be very narrow without hurting
coverage whereas when the two groups are balanced, both intervals are
needed.

In the right panel of Figure \ref{fig:high-dim}, we plot the average
length of prediction intervals against the distances between the two
group means, i.e.,

\[
\left|\xnew^T ( \bbeta_1 - \bbeta_2)\right|.
\]
Recall that in the low-dimensional examples (Figures
\ref{fig:intro_example}, \ref{fig:example_2}, and \ref{fig:example_1}) we observed a marked
decrease in the prediction set length when the means crossed each
other.  A similar phenomenon is apparent here.
There is a linear increasing trend when $\left|\xnew^T ( \bbeta_1 -
  \bbeta_2)\right|$ goes from 0 to 10.  As suggested in the middle panel,
the individual intervals have average length around 5.  Thus, this
linear increase represents the two overlapping intervals gradually
being pulled apart.  At a
distance of 10, they no longer overlap, which explains the leveling of
this trend.  The variability in length seen for distances greater than
10 can be explained, for example, by differing values of $\pi_1(\xnew)$.

\section{Superconductivity Data Application}\label{sec:data}

We apply our method to the superconductivity data provided in
\citet{hamidieh2018data}. This dataset contains the critical
temperature (in Kelvin) and a set of $p=81$ attributes for about 21,000
materials. The attributes used as predictors are elemental property
statistics and electronic structures of attributes. We center and
scale each predictor column and we take the response to be
$\log(1+\text{temperature})$.  The log transform makes the data less
skewed right and adding 1 Kelvin to each temperature can be thought of
as replacing the log of extremely
low temperatures (some are less than $1 mK$) with 0.

We randomly split the observations into a training set of $n=200$ (used
for estimating model parameters, cross validation of
$\lambda_{\balpha}$ and $\lambda_{\bbeta})$, and forming the prediction sets), a
validation set of size 1000 (used to choose $K$), and a test set of
about 20,000 observation (for measuring the coverage
of our prediction sets).  Table \ref{table:MSE_K} shows the mean
squared prediction error, computed on the validation set, for $K$ ranging from 1
to 5.  To make predictions at a given $\x_{\text{validation}}$, we use
$$
\x_{\text{validation}}^T\hbbeta_{\hat k(\x_{\text{validation}})} \qquad\text{where}\qquad
\hat
k(\x_{\text{validation}})=\arg\max_k\hat\pi_k(\x_{\text{validation}})
$$
is the class with highest estimated probability.   The 
prediction errors on the validation set
suggest that $K=2$ may be a suitable choice.

\begin{table}[h!]
	\centering
	\begin{tabular}{p{2cm}|p{2cm}} 
		\hline
		K & Prediction error  \\ 
		\hline
		1 & 8.953 \\ 
		\textbf{2} & \textbf{0.867}\\
		3 & 0.978\\
		4 & 1.318\\
		5 & 0.869\\
		\hline
	\end{tabular}
	\caption{Prediction errors computed on a validation set for fitted models where $K=1,2,3,4,5$.}
	\label{table:MSE_K}
\end{table}

For each observation $(\xnew,\ynew)$ in the test set, we form
$\Omega_{0.05}(\xnew)$ and note whether
$\ynew\in\Omega_{0.05}(\xnew)$.  Figure \ref{fig:sc_100_pred_sets}
displays the prediction sets for a random subset of 100 of the 20,000
intervals formed on the test set. We see
that $\Omega_{0.05}(\xnew)$ is often a single interval, although it
also occasionally the union of two intervals.  The overall coverage on the test set is $97.1\%$.  The average length
of $\Omega_{0.05}(\xnew)$ (after being transformed back from
log-values) is around 42 Kelvin.

\begin{figure}
	\centering
	\scalebox{0.7}{\input{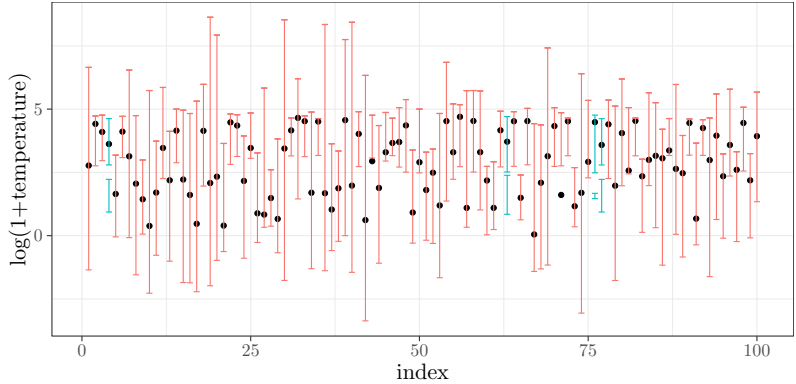}}
	\caption{Plot of prediction sets for $100$ randomly sampled data points in the test set. The black points are the $\log(1+ \text{temperature})$ and the bars correspond to the prediction sets for each observation. We use red when there is only one interval and cyan when there are two intervals.}
	\label{fig:sc_100_pred_sets}
\end{figure}

The $97.1\%$ coverage is averaged over all $\xnew$ and yet our
prediction sets are designed for conditional coverage in the sense of
\eqref{eq:predset}.  This stronger form of coverage implies that we
can get coverage on subsets of observations defined by $\xnew$, i.e.
$$
\P(\ynew\in\Omega_q(\xnew)|\xnew\in\mathcal S)\ge 1-q.
$$
The predictor that is most correlated with critical temperature is the
weighted standard deviation of thermal conductivity.
We divide the range of this variable into 5 equally-spaced
sub-intervals, and divide test data points into 5 subgroups
accordingly. Table \ref{table:CR_thermal} confirms that our prediction
sets meet the nominal level within each subgroup.

\begin{table}[h!]
	\centering
	\begin{tabular}{|c|c|c|} 
		\hline
		Subgroup & Number of data & Coverage Rate  \\ 
		\hline
		1 & 3098 & 95.9\%\\ 
		2 &  4039 & 97.2\%\\
		3 & 3960 & 97.1\% \\
		4 & 3999 & 95.6\% \\
		5 & 4000 & 94.8\% \\
		\hline
	\end{tabular}
	\caption{Coverage rates of $95\%$ prediction sets conditional
          on subgroups defined by the predictor ``weighted standard deviation of thermal conductivity.''}
	\label{table:CR_thermal}
\end{table}

\section{Conclusion}

We have shown how to construct prediction sets for the
high-dimensional mixture of experts model.  Mixture models are
important for capturing the heterogeneity that is present in many
real-world situations.
While in small data samples it was common to dismiss deviations from
the norm as outliers, in large data sets it becomes possible to
use models that can identify and model these subgroups.  While mixture
of regression models allow for such heterogeneity-aware predictive
modeling, they assume that the relative sizes of the subgroups are
fixed.  Importantly, mixture of experts models remove this assumption
and allow the prevalence of different subgroups to
depend on the features.  This generalization is essential in many
situations from 
ecology, where the relative proportions of different subpopulations
depends on environmental covariates \citep{hyun2020}, to politics,
where the political composition depends on demographic and geographic
variables.

Our focus on conditional coverage can be crucial in certain
applications.  For example, \citet{romano2019malice} emphasizes the
importance of ensuring that all subpopulations enjoy the same coverage
guarantees and cast this as a fairness issue when the
subpopulations are defined based on a protected attribute.

\appendix

\section{Proof of Main Theorems}\label{sec:thm-proof}
\subsection{Proof of Theorem~\ref{thm:normality}}
	
	Before we prove the asymptotic normality, we start by presenting a proposition establishing that the optimization problem \eqref{eq:debiase_opt} is feasible. We postpone its proof to Section \ref{section:proof_feasibility}.
	
	\begin{propo}
	\label{lemma:feasibility}(Feasibility)
		Under the assumptions of Theorem \ref{thm:normality}, there exists $\u$ such that\[
		\sup_{\boldsymbol{\omega}\in\mathcal{C}}\left|\left\langle \boldsymbol{\omega},\tilde{\Sigma}_{k}\u-\xnew\right\rangle \right|\le\lambda_{k}\left\Vert \xnew \right\Vert _{2}\text{ and }\left\Vert \u\right\Vert _{1}\le L\twonorm{\xnew}\]
		is satisfied for $\lambda_{k}\asymp \est \log(np) + \sqrt{\log(p)/n}$, $L\ge \frac{\onenorm{\bSigma^{-1}_k\xnew}}{\twonorm{\xnew}}$, and $\mathcal{C}=\left\{ \boldsymbol{e}_{1},\dots,\boldsymbol{e}_{p},\xnew/\left\Vert \xnew\right\Vert _{2}\right\}$.
	\end{propo}

	We next decompose the error of the bias-corrected estimator as follows:
	\begin{align*}
	\hat{\Gamma}^d_k - \Gamma_k
	&=   \xnew^T \hbbeta_k  + \u_k^T  \frac{1}{n}\sum_i \frac{\gamma_{i,k}(\hbtheta)}{\hsigma_k^2} (y_i - \x_i^T \hbbeta_k) \x_i - \xnew^T \bbeta_k\\
	&=  \left(\xnew -  \frac{1}{n}\sum_i \frac{\gamma_{i,k}(\hbtheta)}{\hsigma_k^2}  \x_i \x^T_i \u_k   \right)^T (\hbbeta_k - \bbeta_k) + \u_k^T  \frac{1}{n}\sum_i \frac{\gamma_{i,k}(\hbtheta)}{\hsigma_k^2} ( y_i - \x_i^T \bbeta_k ) \x_i\\
	&=  \left(\xnew - \tilde{\bSigma}_k\u_k \right)^T (\hbbeta_k - \bbeta_k) + \u_k^T  \frac{1}{n}\sum_i \frac{\gamma_{i,k}(\hbtheta)}{\hsigma_k^2} (y_i - \x_i^T \bbeta_k) \x_i.
	\end{align*}
	
	Recall the estimated Fisher information matrix constrained to
        $\bbeta_k$ is given by
\[\SMI=\frac{1}{n}\sum_{i=1}^{n}\frac{\gamma^2_{ik}(\hbtheta)}{\hsigma^2_{k}}\frac{\left(y_{i}-\x_{i}^{T}\hbbeta_{k}\right)^2}{\hsigma_{k}^{2}}\x_{i}\x_{i}^{T},\]
and define
\[\PMI =\mathbb{E}\left[\frac{\gamma^2_{1k}(\btheta)}{\sigma_{k}^{2}}\frac{\left(y_{1}-\x_{1}^{T}\bbeta_{k}\right)^{2}}{\sigma_{k}^{2}}\x_{1}\x_{1}^{T}\right].\]
We then have
\begin{subequations}\label{eq:error_decomp}
	\begin{align*}
	\frac{\sqrt{n}(\hat{\Gamma}_k^d - \Gamma_k)}{\sqrt{\u_k^T \SMI \u_k}}
	& = \frac{\sqrt{n}\left(\xnew - \tilde{\bSigma}_k\u_k \right)^T (\hbbeta_k - \bbeta_k)}{\sqrt{\u^T_k \SMI \u_k} } + \frac{\u_k^T  \frac{1}{\sqrt{n}}\sum_i \frac{\gamma_{i,k}(\hbtheta)}{\hsigma_k^2} (y_i - \x_i^T \bbeta_k) \x_i}{\sqrt{\u^T_k \SMI \u_k} }\\
	& = \underbrace{ \frac{\sqrt{n}\left(\xnew - \tilde{\bSigma}_k\u_k \right)^T (\hbbeta_k - \bbeta_k)}{\sqrt{\u^T_k \SMI \u_k} } }_{(\text{term I})} + \underbrace{ \frac{ \u_k^T \frac{1}{\sqrt{n}} \sum_{i=1}^n \frac{\gamma_{i,k}(\btheta)}{\sigma^2_k} (y_i - \x^T_i \bbeta_k)\x_i + (\text{term II})}{\sqrt{ \u_k^T \SMI \u_k}  }}_{(\text{term III})},\tag{\ref{eq:error_decomp}} 
	\end{align*}
\end{subequations}
	where 
	\[
	\text{term II} = \u_k^T  \frac{1}{n}\sum_i \frac{\gamma_{i,k}(\hbtheta)}{\hsigma_k^2} (y_i - \x_i^T \bbeta_k) \x_i- \u_k^T  \frac{1}{n}\sum_i \frac{\gamma_{i,k}(\btheta)}{\sigma_k^2} (y_i - \x_i^T \bbeta_k) \x_i\,.
\]
	
	We will proceed by stating three propositions which will be
        used to control terms I,II, and III. We defer the proof of these propositions to Section \ref{sec:propos}. 
	
	The first proposition allows us to control the numerator of term I. 

	\begin{propo} \label{lemma:error_term}
	Let $\u_k$ be the solution to optimization \eqref{eq:debiase_opt} with $\lambda_{k}\asymp \est \log(np) + \sqrt{\log(p)/n}$ and $L\ge \onenorm{\bSigma^{-1}_k \xnew}$. Under Assumption (A1) we have
		\begin{align}
		\left|\left(\xnew - \tilde{\bSigma}_k\u_k \right)^T (\hbbeta_k - \bbeta_k)\right| = O_p\left( \left(\est \log(np) + \sqrt{\log(p)/n}\right) \est \twonorm{\xnew}\right).
		\end{align}
	\end{propo}
	
	The second proposition lower bounds the denominator of term I.
	\begin{propo}\label{lemma:fisher_RE}
		
		Under the assumptions of Theorem \ref{thm:normality}, there exist constants $c, C>0$, such that 
		
		\begin{align}
                 \u_k^T \SMI\u_k \geq C \|\xnew\|^2_2\,.
		\end{align}
              \end{propo}

Combining Propositions~\ref{lemma:error_term} and \ref{lemma:fisher_RE}
we obtain

\[
	\text{term I}  = O_p\left( \left(\est \sqrt{n}\log(np) + \sqrt{\log(p)}\right) \est \right) = o_p(1),
	\]
by Condition (A1), equation~\eqref{eq:main-condition}.

The next proposition controls term II.

\begin{propo}\label{lemma:Sigma_conv}
		Under the assumptions of Theorem \ref{thm:normality}, 
		\[
		\left\Vert \frac{1}{n} \sum_{i=1}^n \left( \frac{\gamma_{i,k}(\hbtheta)}{\hsigma_k^2} - \frac{\gamma_{i,k}(\btheta)}{\sigma_k^2}\right) (y_i - \x_i^T \bbeta_k)\x_i \right\Vert_{\infty} = o_p(1).
              \]
	\end{propo}
Therefore, 
\begin{align*}
\text{term II} \leq \left\Vert \frac{1}{n} \sum_{i=1}^n \left( \frac{\gamma_{i,k}(\hbtheta)}{\hsigma_k^2} - \frac{\gamma_{i,k}(\btheta)}{\sigma_k^2}\right) (y_i - \x_i^T \bbeta_k)\x_i \right\Vert_{\infty} \|\u\|_1 = o_p(\|\xnew\|_2)\,.
\end{align*}

The next proposition controls the difference of the sample Fisher
information at the estimated parameter $\SMI$ and the true Fisher information
$\boldsymbol{I}^{\beta}_k (\btheta)$.
 
\begin{propo}\label{lemma:fisher_concentration}
	
	Under the assumptions of Theorem \ref{thm:normality}, and for $\u_k$ the solution of optimization~\eqref{eq:debiase_opt}, we have
	\begin{align}
	\left| \u_k^T \SMI \u_k - \u_k^T \boldsymbol{I}^{\beta}_k (\btheta)\u_k \right| = o_p(\|\xnew\|_2^2).
	\end{align}
\end{propo}
	
We next note that 
\begin{align*}
\nabla_{\bbeta_k} \ell(\btheta)
& =  \frac{1}{n} \sum_{i=1}^n \frac{\gamma_{i,k}(\btheta)}{\sigma^2_k} (y_i - \x^T_i \bbeta_k)\x_i\,,
\end{align*}
 and by asymptotic normality of the score functions, see e.g. \cite{van2000asymptotic}, we have
\[
 \frac{1}{\sqrt{n}} \sum_{i=1}^n \frac{\gamma_{i,k}(\btheta)}{\sigma^2_k} (y_i - \x^T_i \bbeta_k)\x_i \dist N(0, \PMI ).
\]
By sample splitting, $\u_k$ is independent of $\{(\x_i, y_i)\}$ and
therefore is also independent of $\nabla_{\bbeta_k}\ell(\btheta)$.
By invoking Proposition~\ref{lemma:fisher_concentration}, and as an application of Slutsky's Theorem, this implies that
\begin{align}\label{eq:dist-noise}
\frac{ \u_k^T \frac{1}{\sqrt{n}} \sum_{i=1}^n \frac{\gamma_{i,k}(\btheta)}{\sigma^2_k} (y_i - \x^T_i \bbeta_k)\x_i}{\sqrt{ \u_k^T \SMI \u_k} } \dist N(0, 1)\,.
\end{align}

Using the above distributional characterization, together with the bounds on terms I, II, III in decomposition~\eqref{eq:error_decomp}, we obtain the desired result.

%===================================
\section{Proof of
  Theorem~\ref{thm:coverage}}\label{proof:thm:coverage}
Let $\znew$ denote the class of $\xnew$. Conditioning on $\znew$ we have
\begin{align}
\P(\ynew\in\Omega_q(\xnew)|\xnew) &= \sum_{k=1}^K \P(\ynew\in\Omega_q(\xnew)| \znew = k, \xnew) \P(\znew = k |\xnew)\nonumber\\
&=  \sum_{k=1}^K \P(\ynew\in\Omega_q(\xnew)| \znew = k, \xnew) \pi_k(\xnew)\nonumber\\
&=  \sum_{k=1}^K \P(\Gamma_k +  \eps \in\Omega_q(\xnew) |\xnew) \pi_k(\xnew)\,, \label{eq:cov0}
\end{align}
where we recall that $\Gamma_k = \xnew^T\bbeta_k$. Also $\eps\sim
N(0,\sigma_k^2)$ is independent of $\xnew$.

As shown in the proof of Theorem~\ref{thm:normality} we have
\[
\frac{\Gamma_k - \hat{\Gamma}^d_k}{\widehat{V}_k^{1/2}} = W + \Delta\,,
\]
with $W\sim N(0,1)$ and $|\Delta| =o_p(1)$. Therefore,
\begin{align*}
\frac{\Gamma_k +\eps- \hat{\Gamma}^d_k}{b_k} &= \frac{\widehat{V}_k^{1/2}}{b_k} W + \frac{\sigma_k}{b_k}\eps + \frac{\widehat{V}_k^{1/2}}{b_k} \Delta\\
& = c_k W' + \Delta'\,,
\end{align*}
with $W', \eps\sim N(0,1)$ and $|\Delta'| <|\Delta| = o_p(1)$ and $c_k^2: = \frac{\widehat{V}_k+\sigma_k^2}{b_k^2}$. Here we used  the fact that $W,\eps$ are independent normal random variables.

Our procedure for constructing a prediction set returns a union of intervals: $\Omega_q(\xnew) = \cup_{k=1}^K \mathcal{L}_k(q)$ with
$\mathcal{L}_k(q) = [y_k^-,y_k^+]$. Since $\mathcal{L}_k(q) \subseteq \Omega_q(\xnew)$ we have
\begin{align}
\P(\Gamma_k +  \eps \in\Omega_q(\xnew)) &\ge \P(\Gamma_k +  \eps \in\mathcal{L}_k(q))\nonumber\\
&=\P\left(\frac{\Gamma_k +\eps- \hat{\Gamma}^d_k}{b_k} \in \Big[\frac{y_k^- -\hat{\Gamma}^d_k}{b_k},\frac{y_k^+ -\hat{\Gamma}^d_k}{b_k} \Big]\right)\nonumber\\
& = \P\left( c_k W' + \Delta' \in [l_k, u_k] \right)\,,\label{eq:cov1}
\end{align}
where we use the shorthands:
\[
l_k: = \frac{y_k^- -\hat{\Gamma}^d_k}{b_k},\quad u_k:=\frac{y_k^+ -\hat{\Gamma}^d_k}{b_k} \,.
\]
Fix $\gamma>0$ arbitrarily small. We write
\begin{align*}
\P\left( c_k W' + \Delta' \in [l_k, u_k] \right) &\ge \P\left( c_k W'  \in [l_k+\gamma, u_k-\gamma] \right) - \P(|\Delta'|\ge \gamma)\\
& \ge \P\left( W'  \in [(1+\gamma) (l_k+\gamma), (1-\gamma)(u_k-\gamma)] \right) - \P(|\Delta'|\ge \gamma) - \P\Big(\Big|\frac{1}{c_k}-1\Big| \ge \gamma\Big)\,.
\end{align*}
By taking the limit $n\to\infty$ and using the fact that $|\Delta'| = o_p(1)$ and $|\hsigma_k - \sigma_k| = o_p(1)$ (Condition (A1)), we get
\begin{align*}
\lim_{n\to\infty} \P\left( c_k W' + \Delta' \in [l_k, u_k] \right) &\ge \P\left( W'  \in [(1+\gamma) (l_k+\gamma), (1-\gamma)(u_k-\gamma)] \right)  \\
&= \Phi( (1-\gamma)(u_k-\gamma)) - \Phi((1+\gamma) (l_k+\gamma))\,.
\end{align*}
Since $\gamma>0$ was arbitrarily small and the left-hand side does not depend on $\gamma$, by taking $\gamma\to 0$, we arrive at
\begin{align}\label{eq:cov2}
\lim_{n\to\infty} \P\left( c_k W' + \Delta' \in [l_k, u_k] \right) 
 \ge \Phi( u_k) - \Phi(l_k)\,.
\end{align}
Using equations~\eqref{eq:cov1}, \eqref{eq:cov2} in \eqref{eq:cov0} we obtain
\begin{align*}
\lim_{n\to\infty} \P(\ynew\in\Omega_q(\xnew)) &\ge  \sum_{k=1}^K (\Phi( u_k) - \Phi(l_k)) \pi_k(\xnew)\\
& = \sum_{k=1}^K \pi_k(\xnew) \int_{l_k}^{u_k} \phi(t) \de t\\
& = \sum_{k=1}^K  \frac{\hat{\pi}_k(\xnew)}{b_k} \int_{y_k^-}^{y_k^+} \phi\left(\frac{y - \hat{\Gamma}_k^d}{b_k}\right)\de y\\
&= \sum_{k=1}^K   \int_{y_k^-}^{y_k^+} f(y)\de y\\ 
&= \delta \sum_{i=1}^N h_{(i)} -{\rm Err} \\
&\ge (1-q) - {\rm Err}\,,
\end{align*}
where Err is the approximation error for replacing the integral with the Riemann sum. To bound Err we need to upper bound the second derivative of $f(y)$. Define the function $g_a(z) = \frac{1}{a}\phi(\frac{z}{a})$. We have
\[
|g''(z)| = \frac{1}{a^3}\Big|\frac{z^2}{a^2}-1\Big| \phi\Big(\frac{z}{a}\Big) \le \frac{0.18}{a^3}\,.
\] 
Therefore,
\begin{align*}
|f''(y)| = \Big|\sum_{k=1}^K  \frac{\hat{\pi}_k(\xnew)}{b_k} \phi\left(\frac{y - \hat{\Gamma}_k^d}{b_k}\right)\Big|
\le \sum_{k=1}^K \frac{0.18\hat{\pi}_k(\xnew)}{b_k} \le \frac{0.18}{\min_k (b_k)}\,,
\end{align*}
where we used the observation that $\sum_{k=1}^K\hat{\pi}_k(\xnew) = 1$. Therefore, the approximation error over an interval of size $N\delta$ is bounded as
\begin{align}
{\rm Err} \le \frac{0.18}{\min_k (b_k)}\frac{(N\delta)^3}{24N^2} = 0.0075 \frac{N\delta^3}{\min_k (b_k)}\,.
\end{align}
We also note that the returned prediction set is a subset of the initial interval $\mathcal{Q}$ of length ${\rm Len}(\mathcal{Q})$ which implies that $N\delta \le {\rm Len}(\mathcal{Q})$. Hence,
\[
{\rm Err} \le 0.0075 {\rm Len}(\mathcal{Q}) \frac{\delta^2}{\min_k (b_k)} \le \gamma\,,
\]
by our choice of $\delta\le 11 \sqrt{\gamma \min_k (b_k)/{\rm Len}(\mathcal{Q})}$.

\section{Proof of Propositions}\label{sec:propos}
\subsection{Proof of Proposition \ref{lemma:feasibility}} \label{section:proof_feasibility}

	We prove the feasibility of the optimization problem by showing there
	exists $\u$ such that
	\[
	\sup_{\w \in\mathcal{C}}\left|\left\langle \w ,\tilde{\bSigma}_{k}\u-\xnew\right\rangle \right|\le\lambda_{k}\twonorm{\xnew}\text{ and }\left\Vert \u\right\Vert _{1}\le L\twonorm{\xnew}
	\]
	is satisfied for $\lambda_{k}\asymp \est \log(np) + \sqrt{\log(p)/n}$ and $\mathcal{C}=\left\{ \e_{1},\dots,\e_{p},\xnew/\twonorm{\xnew}\right\} $. 
	
	Recall the following quantities:
	\begin{align*}
	\tilde{\bSigma}_{k}&=\frac{1}{n}\sum_{i=1}^{n}\frac{\gamma_{ik}(\hbtheta)}{\hsigma_{k}^{2}}\x_{i}\x_{i}^{T},\\
\bSigma_{k}^{*}&=\frac{1}{n}\sum_{i=1}^{n}\frac{\gamma_{ik}\left(\btheta\right)}{\sigma_{k}^{2}}\x_{i}\x_{i}^{T}, \\
       \bSigma_{k}&= \mathbb{E}\left[\frac{\gamma_{1k}\left(\btheta\right)}{\sigma_{k}^{2}}\x_{1}\x_{1}^{T}\right].
	\end{align*}
	
	Take $\u = \bSigma^{-1}_k\xnew$. By the assumption on $L$ stated in the statement of the proposition, we have $\onenorm{\u}\le L\twonorm{\xnew}$. To show that the other constraint is satisfied we leverage Lemma~\ref{lem:(Covariance-Matrix-Estimation} for the set $\mathcal{C}$ which implies that
	\begin{align*}
	&\sup_{\bxi\in\mathcal{C}} \infnorm{(\tilde{\bSigma}_{k}-\bSigma_{k}^{*})\bxi} = O_{p}\left(\est \log(np)\right)\,,\\
	&\sup_{\bxi\in\mathcal{C}} \infnorm{(\bSigma_{k}-\bSigma_{k}^{*})\bxi} = O_{p}\left(\sqrt{\frac{\log p}{n}}\right).
	\end{align*}
	Combining the above two inequalities we get
	\begin{align}\label{eq:combo}
	\sup_{\bxi\in\mathcal{C}} \infnorm{(\tilde{\bSigma}_{k}-\bSigma_{k})\bxi} = O_{p}\left(\est \log(np) + \sqrt{\log(p)/n}\right)\,.
	\end{align}

                  For any $\w\in\mathcal C$,
	 \begin{align*}
	 \<\w, \tilde{\bSigma}_k\u - \xnew\> &=
	 \<\w, \tilde{\bSigma}_k \u- \bSigma_k \u\>\\
	 &\le \infnorm{(\tilde{\bSigma}_k - \bSigma_k) \w} \onenorm{\u}\\
	 &\le O_{p}\left(\est \log(np) + \sqrt{\log(p)/n}\right) L\twonorm{\xnew}\\
	 &\le \lambda_k \twonorm{\xnew}\,,
	 \end{align*}
	 	where in the second inequality we used~\eqref{eq:combo}.
	 	This completes the proof of the feasibility claim.

\subsection{Proof of Proposition \ref{lemma:error_term}}\label{section:proof_error_term}
As proved in Proposition \ref{lemma:feasibility}, for the choice of $\lambda_{k}\asymp \est \log(np) + \sqrt{\log(p)/n}$ and $L\ge \frac{\onenorm{\bSigma^{-1}_k\xnew}}{\twonorm{\xnew}}$,
optimization problem \eqref{eq:debiase_opt} is feasible. 
The claim follows readily from the constraints of this optimization problem. Specifically,
	\begin{align*}
	\left|\left(\xnew - \tilde{\bSigma}_k\u_k \right)^T (\hbbeta_k - \bbeta_k)\right| & \leq \infnorm{\xnew - \tilde{\bSigma}_k\u_k} \onenorm{\hbbeta_k - \bbeta_k}\\
	&\le \lambda_k \twonorm{\xnew}\onenorm{\hbbeta_k - \bbeta_k}\\
	& \le O_p\left(\Big(\est \log(np) + \sqrt{\log(p)/n}\Big) \est \twonorm{\xnew}\right).
	\end{align*}
	where the first inequality follows from Hölder's inequality (duality of $\ell_1-\ell_\infty$ norms).

%=============================
%===============

\subsection{Proof of Proposition \ref{lemma:fisher_RE}}\label{section:proof_fisher_RE}
We will use the proof strategy of \citet[Lemma
3.1]{javanmard2014confidence}, which was also used in \citet[Lemma 1]{cai2021optimal} and modified to account for the additional constraint~\eqref{eq:debiase_opt2}. However, before doing that we need to deal with the challenge that in optimization~\eqref{eq:debiase_opt} the objective function is based on $\SMI$, while the constraints are in terms of $\tilde{\bSigma}_k$. We first relate $\u_k^T \SMI  \u_k $ to $\u_k^T \tilde{\bSigma}_k\u_k$. 

We denote the first term in \eqref{eq:Q} by $\tQ(\btheta|\hbtheta)$, i.e.,
\begin{align}
\tQ(\btheta|\hbtheta)  = -\frac{1}{n}\sum_{i=1}^n \sum_{k=1}^K \gamma_{i,k}(\hbtheta) \left[\log \pi_{k}(\x_i)  + \log\phi_{k}(\x_i, y_i) \right]\,.
\end{align}
We restate a lemma from \citet{wang2014high} that allows us to connect the derivatives of the function $\tQ(\btheta|\hbtheta)$ with the derivatives of the log-likelihood $\ell(\btheta)$.

\begin{lemma}\cite[Lemma 2.1]{wang2014high} \label{lem:Q}
For the true parameter $\btheta$ and any $\tilde{\btheta}$, it holds that
\begin{align}
\nabla_1 \tilde{Q}(\tilde{\btheta} | \tilde{\btheta}) &= \nabla \ell(\tilde{\btheta})\,,\label{eq:Q-grad}\\
\E[\nabla^2_{1,1} \tilde{Q}(\btheta | \btheta) + \nabla^2_{1,2} \tilde{Q}(\btheta | \btheta)] &= - {\bf I}(\btheta) \,,\label{eq:Q-hess}
\end{align}
where ${\bf I}(\btheta) = -\E[\nabla^2\ell(\btheta)]$ is the Fisher
information matrix and $\nabla_1$ and $\nabla_2$ denote
differentiation with respect to $\bbeta_k$ in the first and second arguments of
$\tilde{Q}(\cdot |\cdot)$, respectively.
\end{lemma}

 Following the same argument as in the proof of \cite[Lemma 3]{zhang2020}, we have 
 \begin{align}\label{eq:Hess12-psd}
 -\nabla_{1,2}^2 \tilde{Q}(\btheta|\btheta)= -\frac{\partial^2}{\partial \bbeta_k \partial \bbeta_k'} \tilde{Q}(\btheta|\btheta')|_{\btheta' = \btheta} \succeq 0\,.
 \end{align}
Combining~\eqref{eq:Q-hess} and \eqref{eq:Hess12-psd} we get
  \begin{align}
\PMI &=  -\E[\nabla^2_{1,1} \tilde{Q}(\btheta | \btheta) + \nabla^2_{1,2} \tilde{Q}(\btheta | \btheta)]\nonumber\\
 &\succeq -\E[\nabla^2_{1,1} \tilde{Q}(\btheta | \btheta)]\nonumber\\
 & = \E\left[\frac{\gamma_{1k}(\btheta)}{\sigma_k^2} \x_1\x_1^T\right] =: \bSigma_k\,,\label{eq:LB1}
  \end{align}
where in the last step we used the notations defined in Lemma~\ref{lem:(Covariance-Matrix-Estimation}.
 
 We next write the following chain of terms:
 \begin{align}
 \u_k^T \SMI \u_k  =  &\u_k^T \tilde{\bSigma}_k \u_k + \u_k^T (\PMI - \bSigma_k)\u_k\nonumber\\
 &+ \u_k^T(\bSigma_k- \bSigma_k^*)\u_k + \u_k^T(\bSigma^*_k- \tilde{\bSigma}_k)\u_k\nonumber\\
 &+\u_k^T (\SMI -  \PMI)\u_k\nonumber\\
 &\ge \u_k^T \tilde{\bSigma}_k \u_k + \u_k^T (\PMI - \bSigma_k)\u_k - O_p\left(\|\xnew\|_2^2\left\{\log(pn)\sqrt{\frac{\log
           p}{n}}+\est\sqrt{\log^3(pn)\log(nK)}\right\}\right)\nonumber\\
 &\ge \u_k^T \tilde{\bSigma}_k \u_k-o_p(\|\xnew\|_2^2)\,,\label{eq:I-Sig}
 \end{align}
 where in the first inequality we used
 Lemma~\ref{lem:(Covariance-Matrix-Estimation} and Proposition \ref{lemma:fisher_concentration}, together with the fact that
 $\onenorm{\u}\le L\|\xnew\|_2$.  The last step in~\eqref{eq:I-Sig} follows from the condition $\sqrt{n} \log(np)  \est^2 = o(1)$ according to Assumption (A1), and the assumption $\log(p)= o(n^{1/4}/\sqrt{\log(n)})$.
 
We next lower bound $\u_k^{T}\tilde{\bSigma}_{k}\u_k$. For any feasible solution $\u$ of optimization \eqref{eq:debiase_opt} we have:
	\begin{align}\label{eq:facilitating}
	\u^T \tilde{\bSigma}_k \u \geq \u^T \tilde{\bSigma}_k \u  + t((1-\lambda_k) \|\xnew\|^2_2 - \xnew^T \tilde{\bSigma}_k \u),
	\end{align}
	for any $t>0$. The last inequality holds true because by the constraint of optimization \eqref{eq:debiase_opt}, we have
	\[
	\|\xnew\|_2^2 - \xnew^T \tilde{\bSigma}_k \u \leq  \left|\xnew^T \tilde{\bSigma}_k \u - \|\xnew\|^2_2 \right| \leq \|\xnew\|_2^2 \lambda_k.
	\]
	
Minimizing over all feasible $\u$ gives
	\begin{align*}
	\u_k^T\tilde{\bSigma}_k\u_k \ge \min_{\u}\left\{\u^T \tilde{\bSigma}_k \u  + t((1-\lambda_k) \|\xnew\|^2_2 - \xnew^T \tilde{\bSigma}_k \u)\right\}\,.
	\end{align*}
	The minimizer $\u^*$ satisfies $\tilde{\bSigma}_k\u^* =
        \frac{t}{2} \tilde{\bSigma}_k\xnew$. Substituting for $\u^*$, we obtain
	\[
	\u_k^T\tilde{\bSigma}_k\u_k \ge -\frac{t^2}{4} \xnew^T \tilde{\bSigma}_k \xnew+ t(1-\lambda_k) \twonorm{\xnew}^2\,.
	\]
	Optimizing this bound over $t$, we get 
	\[
	\u_k^{T}\tilde{\bSigma}_k \u_k \geq \frac{(1-\lambda_k)^2\twonorm{\xnew}^4}{\xnew^T\tilde{\bSigma}_k \xnew},
	\]
	with the optimal choice $t^* = \frac{2(1-\lambda_k) \twonorm{\xnew}^2}{\xnew^T \tilde{\bSigma}_k\xnew} > 0$.
	
	Similar to the proof of Lemma \ref{lem:(Covariance-Matrix-Estimation}, \ref{lem:(Fisher-Information-Estimation}, we have $\left| \frac{ \xnew^T \tilde{\bSigma}_k \xnew}{ \xnew^T \bSigma_k \xnew} - 1\right| = O_p(\sqrt{(\log p)/n})$, and hence $\u_k^{T}\tilde{\bSigma}_k \u_k \ge C\twonorm{\xnew}^2$, which in conjunction with equation~\eqref{eq:I-Sig} gives the desired result.

%===============
\subsection{Proof of Proposition \ref{lemma:Sigma_conv}}
\label{section:proof_Sigma_conv}
By the triangle inequality we have
\begin{align}\label{eq:tri}
\left|\frac{\gamma_{i,k}(\hbtheta)}{\hsigma_k^2} - \frac{\gamma_{i,k}(\btheta)}{\sigma_k^2}\right|
\le \left| \frac{\gamma_{i,k}(\hbtheta)}{\hsigma_k^2} - \frac{\gamma_{i,k}(\hbtheta)}{\sigma_k^2}\right|
+\frac{|\gamma_{i,k}(\hbtheta)- \gamma_{i,k}(\btheta)|}{\sigma_k^2} \,.
\end{align}
The first term can be bounded as follows by recalling Condition (A1) on the error term $|\hsigma_k -\sigma_k|$:
\[
\left| \frac{\gamma_{i,k}(\hbtheta)}{\hsigma_k^2} - \frac{\gamma_{i,k}(\hbtheta)}{\sigma_k^2}\right|
= \frac{\gamma_{i,k}(\hbtheta)}{\sigma_k^2} \cdot \left|\frac{\sigma_k^2}{\hsigma_k^2} -1 \right|
\le \frac{1}{\sigma_k^2} \cdot \left|\frac{\sigma_k^2}{\hsigma_k^2} -1 \right| = O(|\sigma_k - \hsigma_k|) = O_p(\est)\,,
\]
uniformly over all $i\in[n]$, $k\in[K]$. For the second term, by using Lemma \ref{lemma:lipschitz_gamma} together with Assumption (A1) on the error term $\onenorm{\hbtheta - \btheta}$ we get
\[
\sup_{i\in[n],k\in[K]} \frac{|\gamma_{i,k}(\hbtheta)- \gamma_{i,k}(\btheta)|}{\sigma_k^2} = O_p(\sqrt{\log(np)\log(nK)}\est)\,.
\] 
Combining the above two bounds into~\eqref{eq:tri} we get that
\begin{align}\label{eq:tri2}
\sup_{i\in[n],k\in[K]} \left|\frac{\gamma_{i,k}(\hbtheta)}{\hsigma_k^2} - \frac{\gamma_{i,k}(\btheta)}{\sigma_k^2}\right| =O_p(\sqrt{\log(np)\log(nK)}\est)\,.
\end{align}
Next, we observe that
\begin{align*}
  \infnorm{ (y_i - \x_i^T \bbeta_k)\x_i }&\le\max_{ik}|y_i
                                           -\x_i^T\bbeta_k|\cdot\max_{i}\|x_{i}\|_\infty=O_p(\sqrt{\log(nK)}\cdot\sqrt{\log(np)})
\end{align*}
by Lemma \ref{lem:events}.  Using this result combined with
\eqref{eq:tri2} gives the bound
\begin{align*}
&\infnorm{ \frac{1}{n} \sum_{i=1}^n \left( \frac{\gamma_{i,k}(\hbtheta)}{\hsigma_k^2} - \frac{\gamma_{i,k}(\btheta)}{\sigma_k^2}\right) (y_i - \x_i^T \bbeta_k)\x_i }\\
& \le \sup_{i\in[n],k\in[K]} \left|  \frac{\gamma_{i,k}(\hbtheta)}{\hsigma_k^2} - \frac{\gamma_{i,k}(\btheta)}{\sigma_k^2} \right|
\cdot\frac{1}{n}\sum_{i=1}^n \infnorm{ (y_i - \x_i^T \bbeta_k)\x_i }\\
&=O_p\Big(\log(np) \log(nK) \est\Big)\\
&=o_p\Big(\sqrt{\log(np)} \cdot {\log(nK)}\cdot n^{-1/4}\Big),
\end{align*}
where the last line follows from our assumption on the estimation rate $\est$, cf. equation~\eqref{eq:main-condition}. Finally note that by invoking the assumption $\log(p)= o(n^{1/4}/\sqrt{\log(n)})$, the last term is $o_p(1)$.

%===============
\subsection{Proof of Proposition \ref{lemma:fisher_concentration}}\label{section:proof_fisher_concentration}

	Note that for a matrix $\mathbf{A}$ and a vector $\u$ we have
	\[
	|\u^T \mathbf{A} \u| = |\sum_{ij} A_{ij} u_j u_j|
	\le \infMnorm{\mathbf{A}} \sum_{ij} |u_i| |u_j| = \infMnorm{\mathbf{A}}\onenorm{\u}^2\,. 
	\]

 From Lemma
        \ref{lem:(Fisher-Information-Estimation},
        $|\SMI-\boldsymbol{I}^{\beta}_k
        (\btheta)|_\infty=O_{p}\left(\sqrt{\log^3(np)\log(nK)}\est+\log(np)\sqrt{\frac{\log(p)}{n}}\right)$.  
                
        Given that $\onenorm{\u_k}\le L\|\xnew\|_2$ for a constant
        $L$, and $\sqrt{n} \log(np)  \est^2 = o(1)$ according to Assumption (A1) we have

        $$
        \left| \u_k^T \SMI \u_k - \u_k^T \boldsymbol{I}^{\beta}_k (\btheta)\u_k \right|=\|\xnew\|_2^2\cdot O_{p}\left(\log(np)\sqrt{\log(nK)}n^{-1/4}+\log(np)\sqrt{\frac{\log(p)}{n}}\right)=o_p(\|\xnew\|_2^2),
        $$
        where the last step follows from the assumption $\log(p) = o(n^{1/4}/\sqrt{\log(n)})$.

%******************************************
\section{Intermediate Lemmas and Proofs}
This section summarizes several technical lemmas that were used in establishing our theoretical results.

The first lemma is a classical maximal inequality for sub-Gaussian random  variables. 
\begin{lemma} \label{lemma:max-gaussians-bound}
	(Maximal Inequality for Sub-Gaussians) Let $\x=(x_{1},\dots,x_{n})$ be
	a vector of zero-mean sub-Gaussian random variables with
	variances $v_{1}^{2},\dots,v_{n}^{2}$, respectively. Then, for any constant $c>0$  we have 
	\[
	\sup_{i=1,\dots,n}\left|x_{i}\right|\le v_{\max}\sqrt{2c \log n},
	\]
	with probability
	at least $1-2n^{1-c}$,
	where $v_{\max} = \max_{i\in[n]} v_i$.
\end{lemma}
The above result is obtained by using the tail bound of sub-Gaussian
variables, followed by a simple union bound, and therefore it does not
require the random variables to be independent.

Our next lemma is a simple corollary of the above maximal inequality.
\begin{lemma}\label{lem:events}
Define the following probability events:
\begin{align*}
\event_1 &:= \left\{\max_{i\in [n]} \infnorm{\x_i}\le C\sqrt{\log(np)} \right\},\\
\event_2 &:= \left\{\max_{i\in[n], k\in [K]} \;  |y_i - \x_i^T\bbeta_k| \le C \sqrt{\log(nK)} \right\}\,,
\end{align*}
and let $\event: = \event_1\cap \event_2$. Then,  under Assumptions (A2)--(A3) we have $\P(\event) \ge 1 - \frac{4}{n^2}$, for large enough constant $C>0$.
\end{lemma}
\begin{proof}
we write $x_{ir} = \<\bSigma^{1/2} \e_r, \bSigma^{-1/2}\x_i\>$. Therefore,
\begin{align}\label{eq:xir-sub-G}
\|x_{ir}\|_{\psi_2} \le\twonorm{\bSigma^{1/2} \e_r} \|\bSigma^{-1/2}\x_i \|_{\psi_2} \le C_{\Sigma}^{1/2} \|\bSigma^{-1/2}\x_i \|_{\psi_2} < \kappa\,,
\end{align}
for all $r\in[p]$, $i\in[n]$ and some constant $\kappa<\infty$ by Assumption (A2). By Lemma~\ref{lemma:max-gaussians-bound}, and for sufficiently large constant $C>0$, we get $\P(\event_1)\ge 1 - \frac{2}{n^2}$.

To bound probability of  $\event_2$, suppose that sample $i$ belongs to group $\ell$, by which we can write
\[
y_i -\x_i^T\bbeta_k = \x_i^T(\bbeta_\ell - \bbeta_k) + \eps_i\,.
\]
We then have
\begin{align*}
\|y_i -\x_i^T\bbeta_k\|_{\psi_2} &\le \max_{\ell\in[K]} \|\x_i^T(\bbeta_\ell - \bbeta_k) + \eps_i\|_{\psi_2}\\
&\le \max_{\ell\in[K]} \left( \|\x_i^T(\bbeta_\ell - \bbeta_k)\|_{\psi_2} +  \|\eps_i\|_{\psi_2} \right)\\
&\le\max_{\ell\in[K]} \left( \twonorm{\bSigma^{1/2}(\bbeta_\ell - \bbeta_k)}\|\bSigma^{-1/2}\x_i\|_{\psi_2} +  \sigma_\ell \right)\\
&\le  \max_{\ell\in[K]} \left(C_{\Sigma} \twonorm{\bbeta_\ell - \bbeta_k} \|\bSigma^{-1/2}\x_i\|_{\psi_2} +  \sigma_\ell \right)\,.
\end{align*}
 Taking maximum over $i,k$ from both sides, we get
\begin{align}\label{eq:res-X2}
\sup_{i\in[n], k\in[K]} \|y_i -\x_i^T\bbeta_k\|_{\psi_2} \le 
C_\Sigma\left(\max_{i\in[n]} \|\bSigma^{-1/2}\x_i\|_{\psi_2} \right) \left( \max_{k, \ell\in[K]} \twonorm{\bbeta_\ell - \bbeta_k} \right) + \max_{k\in[K]} \sigma_k\,.
\end{align}
Recalling Assumptions (A2) and (A3), we get $\sup_{i\in[n], k\in[K]} \|y_i -\x_i^T\bbeta_k\|_{\psi_2} \le \kappa'$, for some constant $\kappa'<\infty$. Therefore, by another application of Lemma~\ref{lemma:max-gaussians-bound}, and for sufficiently large constant $C>0$, we get $\P(\event_2)\ge 1 - \frac{2}{n^2}$.

Combining the two probability bounds we get $\P(\event) \ge 1 - \P(\event_1) - \P(\event_2) \ge 1 - \frac{4}{n^2}$.
\end{proof}

While Assumption (A1) concerns the estimation error $\onenorm{\hbtheta
  - \btheta}$, in our analysis we often need to control the
perturbation of different functions of $\btheta$. A useful step for
these bounds is a control on the Lipschitz factor of
$\gamma_{i,k}(\btheta)$, which is the subject of the next lemma.

\begin{lemma}
	\label{lemma:lipschitz_gamma}{(Lipschitzness of
          $\gamma_{i,k}(\btheta)$)} On event $\event$, defined in
        Lemma~\ref{lem:events}, the Lipschitz factor of
        $\gamma_{i,k}(\btheta)$ with respect to the $\ell_1$ norm is $O_p(\log(np))$, uniformly over all $i\in [n]$, $k\in [K]$. As a result,
	\[
	\sup_{i\in[n], k\in[K]} \frac{|\gamma_{i,k}(\hbtheta) - \gamma_{i,k}(\btheta)|}{\onenorm{\hbtheta - \btheta}} = O_p(\sqrt{\log(np) \log(nK)})\,.
	\]	
\end{lemma}

\begin{proof}
	
	To prove the claim, it suffices to show that
	
	\[
	\sup_{i\in[n],k\in[K]} \left(\left\Vert \frac{\partial \gamma_{i,k}}{\partial \bbeta_k}\right\Vert_{\infty}, \left\Vert \frac{\partial \gamma_{i,k}}{\partial \balpha_k}\right\Vert_{\infty}, \left| \frac{\partial \gamma_{i,k}}{\partial \sigma_k}\right|\right) = O_p(\sqrt{\log(np) \log(nK)}).
	\]
	
	Recall that $\gamma_{i,k}(\btheta) = \frac{\pi_{k}(\x_i) \phi_{k}(\x_i, y_i)}{\sum_{\ell = 1}^K \pi_{\ell}(\x_i) \phi_{\ell}(\x_i, y_i)}$, with $\phi_{k}$ given by~\eqref{eq:phi-def}. A simple algebraic calculation shows that for a function of form $f_k(\z) = \frac{c_k e^{z_k}}{\sum_{\ell} c_\ell e^{z_\ell}}$, we have 
	\begin{align}\label{eq:f-aux}
	\frac{\partial}{\partial z_k} f_k = f_k(z) - f_k^2(z)\,.
	\end{align}
	 Applying this result, we obtain
	\[
	\frac{\partial \gamma_{i,k}}{\partial \bbeta_k}  = (\gamma_{i,k} - \gamma^2_{i,k}) \frac{1}{\sigma_k^2} (y_i - \x^T_i \bbeta_k) \x_i.
	\]
	Therefore, on the event $\event$ we have 
	
	\begin{align*}
	\infnorm{\frac{\partial \gamma_{i,k}}{\partial \bbeta_k} } &\leq \frac{1}{\sigma_k^2} \left| \gamma_{i,k} - \gamma^2_{i,k}\right| |y_i - \x^T_i \bbeta_k| \cdot \infnorm{\x_i}\\
	&\leq \frac{1}{4\sigma_k^2} C^2 \sqrt{\log(np) \log(nK)} = O(\sqrt{\log(np) \log(nK)}),
	\end{align*}
	where we used the definition of event $\event$ and the fact that since $0 \leq \gamma_{i,k} \leq 1$, $\gamma_{i,k}(1 - \gamma_{i,k}) \leq \frac{1}{4}$. 
	
	Similarly, we can bound the partial derivative with respect to $\balpha_k$. By another application of \eqref{eq:f-aux}, we obtain 

	\begin{align*}
	\frac{\partial \gamma_{i,k}}{\partial \balpha_k}  &= (\gamma_{i,k} - \gamma^2_{i,k}) \x_i\,,\\
	\infnorm{\frac{\partial \gamma_{i,k}}{\partial \balpha_k} } &\leq |\gamma_{i,k} - \gamma^2_{i,k}|\; \infnorm{\x_i} \leq \frac{1}{4}\infnorm{\x_i} = O(\sqrt{\log(np)})\,,\\
	\sup_{i\in[n], k\in[K]} \left\Vert	\frac{\partial \gamma_{i,k}}{\partial \balpha_k} \right\Vert_{\infty} &\leq \sup_i \|\x_i\|_{\infty} \leq O(\sqrt{\log(np)})\,.
	\end{align*}
	
	Finally, we bound the partial derivative with respect to $\sigma_k$. By another application of \eqref{eq:f-aux}, we obtain 
	
	\begin{align*}
	\left|\frac{\partial \gamma_{i,k}}{\partial \sigma_k} \right| &= (\gamma_{i,k}
          - \gamma^2_{i,k}) \frac{1}{\sigma_k^3}
          (y_i-\x_i^T\bbeta_k)^2 = O(\log(nK))\,.
	\end{align*}
This completes the proof.
	
\end{proof}

\begin{lemma}
	\label{lemma:mixture_estimation_error}(Mixture Estimation Error)
	On the event $\event$, defined in Lemma~\ref{lem:events}, we have
	\[
	\sup_{i\in[n],k\in[K]}\left|\frac{\gamma_{ik}(\hbtheta)}{\hsigma_{k}^2}-\frac{\gamma_{ik}(\btheta)}{\sigma_{k}^{2}}\right|=O_p\left(\sqrt{\log(np)\log(nK)} \est \right).
	\]
\end{lemma}
\begin{proof}
	We have that 
	\[
	\sup_{i\in[n],k\in[K]} \left| \frac{\gamma_{ik}(\hbtheta)}{\hsigma_{k}^{2}}-\frac{\gamma_{ik}(\btheta)}{\sigma_{k}^{2}}\right|\le\underbrace{\sup_{i,k} {\gamma_{ik}(\hbtheta)}\left|\frac{1}{\hsigma_k^2}-\frac{1}{\sigma_{k}^2}\right|}_{(a)}+\underbrace{\sup_{i,k}\left|\frac{\gamma_{ik}(\hbtheta)-\gamma_{ik}(\btheta)}{\sigma_{k}^{2}}\right|}_{(b)}.
	\]
	For $(a)$, we have 
	\[
	(a) \le  \sup_{k} \left|\frac{1}{\hsigma_k^2}-\frac{1}{\sigma_{k}^2}\right| = \sup_k \frac{|\hsigma_k - \sigma_k|\; |\hsigma_k + \sigma_k|}{\sigma_k^2 \hsigma_k^2} =O_{p}\left(\est\right)\,,
	\]
	by Assumption (A1) and (A3). For $(b)$, as shown in Lemma \ref{lemma:lipschitz_gamma} we have
	\[
	\sup_{i,k}\left|\frac{\gamma_{ik}(\hbtheta)-\gamma_{ik}(\btheta)}{\sigma_{k}^{2}}\right|\le\frac{1}{\sigma_{k}^{2}} O_p(\sqrt{\log(np) \log(nK)}\|\hbtheta-\btheta\|_1) = O_{p}\left(\sqrt{\log(np) \log(nK)} \est\right),
	\]
	where in the last step we used Assumption (A1) and (A3).
\end{proof}

The next lemma is a concentration result on the covariate vectors $\x_i$ which will be used in our analysis.

\begin{lemma}
	\label{lem:(Concentration-of-Design} Let $\{\x_{i}\}_{i=1}^n$ satisfy Assumption (A2). Then for any fixed unit vector $\bxi$ we have 
	\[
	\P\left\{\sup_{\ell\in[p]} \left|\frac{1}{n}\sum_{i=1}^{n}\left|x_{i,\ell} \<\x_{i},\bxi\>\right|-\mathbb{E}\left[\left|x_{i,\ell} \<\x_i,\bxi\>\right|\right]\right|
	\ge t\right\} \le 2 p \exp\left( -c \min\left( 
	\frac{n t^2}{C^2},  
	\frac{n t}{C}
	\right) \right)\,,
	\]
	and 
	\[\mathbb{E}\left[|x_{i,\ell} \<\x_i,\bxi\>| \right]\le C,
	\]
	for some positive constant $C$.
\end{lemma}

\begin{proof}
	To obtain the first result, we first note that $\left|x_{i,\ell} \<\x_i,\bxi\>\right|$
	is a product of two sub-Gaussian random variables and thus is a sub-exponential
	variable. Applying Theorem 2.8.1 of \citet{vershynin2018high} to a fixed $\ell\in[p]$, we get that for $C\ge \max_{i\in [n]}  \left\Vert x_{i,\ell} \<\x_i,\bxi\> \right\Vert_{\psi_1}$,
	
	\begin{align}\label{eq:probB}
	\mathbb{P} \left\{ \left| 
	\frac{1}{n}\sum_{i=1}^{n} \left|x_{i,\ell} \<\x_i,\bxi\>\right|-\mathbb{E}\left[\left|x_{i,\ell} \<\x_i,\bxi\>\right|\right] \right| \ge t \right\}
	\le 2 \exp\left( -c \min\left( 
	\frac{nt^2}{ C^2},  
	\frac{n t}{C}
	\right) \right),
	\end{align}
	where $\|\cdot\|_{\psi_1}$ is the sub-exponential norm of a random variable and $c>0$ is an absolute constant.
	
	To obtain our intended bound, we compute the sub-exponential norms 
	$\left\Vert x_{i,\ell} \<\x_{i},\bxi \> \right\Vert_{\psi_1}$. By Lemma 2.7.7 of \citet{vershynin2018high}, we have
	\begin{align}
		 \|x_{i,\ell}\<\x_{i},\bxi\>\|_{\psi_1} &\le \|x_{i,\ell}\|_{\psi_2} \|\<\x_{i},\bxi\>\|_{\psi_2}\nonumber\\
		 &=  \|\<\bSigma^{1/2} \e_\ell,\bSigma^{-1/2}\x_i\> \|_{\psi_2}
		 \|\<\bSigma^{1/2} \bxi,\bSigma^{-1/2}\x_i\> \|_{\psi_2}\nonumber\\
		 &\le \twonorm{\bSigma^{1/2} \e_\ell}\twonorm{\bSigma^{1/2} \bxi}
		  \|\bSigma^{-1/2}\x_i\|_{\psi_2}^2\nonumber\\
		 &\le C' C_{\Sigma} =: C\,,\label{eq:sub-exp-1}
	\end{align}
	where we used that $\|\bSigma\|_{{\rm op}} \le C_{\Sigma}$ and 
	$ \|\bSigma^{-1/2}\x_i\|_{\psi_2} \le C'$, for some constants
        $C', C_{\Sigma}$, per Assumption (A2).

To prove the second part of the lemma, we note that by definition
\[
\|x_{i,\ell} \<\x_i,\bxi\> \|_{\psi_1} = \sup_{q\ge 1} q^{-1} \E[|x_{i,\ell} \<\x_i,\bxi\> |^q]^{1/q}\,,
\]
by which we have
\[
\E[|x_{i,\ell} \<\x_i,\bxi\> |] \le \|x_{i,\ell} \<\x_i,\bxi\> \|_{\psi_1} \le C\,,
\]
for all $\ell\in[p]$.
\end{proof}

\begin{lemma}
	\label{lem:(Covariance-Matrix-Estimation} (Covariance Matrix Estimation
	Error) Let \begin{align*}
	\tilde{\bSigma}_{k}&=\frac{1}{n}\sum_{i=1}^{n}\frac{\gamma_{ik}(\hbtheta)}{\hsigma_{k}^{2}}\x_{i}\x_{i}^{T},\\
\bSigma_{k}^{*}&=\frac{1}{n}\sum_{i=1}^{n}\frac{\gamma_{ik}\left(\btheta\right)}{\sigma_{k}^{2}}\x_{i}\x_{i}^{T}, \\
       \bSigma_{k}&= \mathbb{E}\left[\frac{\gamma_{1k}\left(\btheta\right)}{\sigma_{k}^{2}}\x_{1}\x_{1}^{T}\right].
	\end{align*}
	Consider a set of unit-norm vectors $\mathcal{C} = \{\bxi_1,\dotsc, \bxi_{p^m}\}$, for a fixed integer $m\ge 1$. Under Assumption (A2), and on the event $\event$ defined in Lemma~\ref{lem:events}, we have 
	\[
	\sup_{\bxi\in\mathcal{C}} \infnorm{(\tilde{\bSigma}_{k}-\bSigma_{k}^{*})\bxi} = O_{p}\left(\est \sqrt{\log(np) \log(nK)}\right)\,,
	\]
	and
	\[
	\sup_{\bxi\in\mathcal{C}} \infnorm{(\bSigma_{k}-\bSigma_{k}^{*})\bxi} = O_{p}\left(\sqrt{\frac{\log p}{n}}\right).
	\]
\end{lemma}

\begin{proof}
	We have 
	\begin{align*}
	\sup_{\bxi\in\mathcal{C}} \infnorm{(\tilde{\bSigma}_{k}-\bSigma_{k}^{*})\bxi} & \le\sup_{\ell\in[p], \bxi\in\mathcal{C}} \frac{1}{n}\sum_{i=1}^{n}\left|\frac{\gamma_{ik}(\hbtheta)}{\hsigma_{k}^{2}}-\frac{\gamma_{ik}(\btheta)}{\sigma_{k}^{2}}\right|\left|x_{i,\ell}\right|\left|\<\x_i,\bxi\>\right|\\
	& \le\sup_{i\in[n]}\left|\frac{\gamma_{ik}(\hbtheta)}{\hsigma_{k}^{2}}-\frac{\gamma_{ik}(\btheta)}{\sigma_{k}^{2}}\right| \cdot \sup_{\ell\in[p], \bxi\in\mathcal{C}} \frac{1}{n}\sum_{i=1}^{n}\left|x_{i,\ell}\right|\left|\<\x_i,\bxi\>\right|.
	\end{align*}
	For the first component, we leverage Lemma \ref{lemma:mixture_estimation_error}, which proves a $O_p(\est\sqrt{\log(np)\log(nK)})$ bound.
	For the second component, we
	use Lemma \ref{lem:(Concentration-of-Design}  with $t = C\sqrt{\frac{2m\log(p)}{c n}}$ and union bound over the set $\mathcal{C}$, which gives 
	\[
	\sup_{\ell\in[p], \bxi\in\mathcal{C}} \frac{1}{n}\sum_{i=1}^{n}\left|x_{i,\ell}\right|\left|\<\x_i,\bxi\>\right| =O_{p}\left(1+ \sqrt{\frac{\log p}{n}}\right) = O_p(1).
	\]
	Putting them together, we have
	\begin{align*}
	\sup_{\bxi\in\mathcal{C}} \infnorm{(\tilde{\bSigma}_{k}-\bSigma_{k}^{*})\bxi}&=O_{p}\left(\est\sqrt{\log(np) \log(nK)}\right) \,.
	\end{align*}

	To prove the second result, we have
	\[
	\sup_{\ell\in[p], \bxi\in\mathcal{C}} \left\Vert (\bSigma_{k}^{*}-\bSigma_{k})\bxi \right\Vert _{\infty}\le\sup_{\ell\in[p], \bxi\in\mathcal{C}} \left|\frac{1}{n}\sum_{i=1}^{n}\left\{\frac{\gamma_{ik}(\btheta)}{\sigma_{k}^{2}}x_{i,\ell}\<\x_i,\bxi\>-\mathbb{E}\left[\frac{\gamma_{ik}(\btheta)}{\sigma_{k}^{2}}x_{i,\ell} \<\x_i,\bxi\>\right]\right\}\right|.
	\]
	As shown in~\eqref{eq:sub-exp-1}, $x_{i,\ell}\<\x_i,\bxi\>$ have bounded sub-exponential norm and since $\frac{\gamma_{ik}(\btheta)}{\sigma_{k}^{2}}$ is bounded, we have that $\frac{\gamma_{ik}(\btheta)}{\sigma_{k}^{2}}x_{i,\ell}\<\x_i,\bxi\>$
	has bounded sub-exponential norm. Therefore, by Theorem 2.8.1 of \citet{vershynin2018high} (similar to Lemma~\ref{lem:(Concentration-of-Design})
	and a union bound over $\ell\in[p]$ and $\bxi\in\mathcal{C}$, we obtain,
	\[
	\sup_{\ell\in[p], \bxi\in\mathcal{C}} \left|\frac{1}{n}\sum_{i=1}^{n}\left\{\frac{\gamma_{ik}(\btheta)}{\sigma_{k}^{2}}x_{i,\ell}\<\x_i,\bxi\>-\mathbb{E}\left[\frac{\gamma_{ik}(\btheta)}{\sigma_{k}^{2}}x_{i,\ell} \<\x_i,\bxi\>\right]\right\}\right|=O_{p}\left(\sqrt{\frac{\log p}{n}}\right),
	\]
	which completes the proof of the second part.
\end{proof}

%====================
\begin{lemma}
	\label{lem:(Fisher-Information-Estimation} (Fisher Information Estimation Error) Let 
	\begin{align*}
	\SMI &=\frac{1}{n}\sum_{i=1}^{n}\frac{\gamma^2_{ik}(\hbtheta)}{\hsigma^2_{k}}\frac{\left(y_{i}-\x_{i}^{T}\hbbeta_{k}\right)^2}{\hsigma_{k}^{2}}\x_{i}\x_{i}^{T},\\
	\SPMI &=\frac{1}{n}\sum_{i=1}^{n}\frac{\gamma^2_{ik}(\btheta)}{\sigma^2_{k}}\frac{\left(y_{i}-\x_{i}^{T}\bbeta_{k}\right)^2}{\sigma_{k}^{2}}\x_{i}\x_{i}^{T},\\
	\PMI &=\mathbb{E}\left[\frac{\gamma^2_{1k}(\btheta)}{\sigma_{k}^{2}}\frac{\left(y_{1}-\x_{1}^{T}\bbeta_{k}\right)^{2}}{\sigma_{k}^{2}}\x_{1}\x_{1}^{T}\right].
	\end{align*}
	Under Assumptions (A1), (A2) and (A3), We have
	\[
	\infMnorm{\SMI-\SPMI}=O_{p}\left(\sqrt{\log^3(np)\log(nK)}\est\right),
	\]
	and
	\[
	\infMnorm{\SPMI-\PMI}=O_p\left(\log(np)\sqrt{\frac{\log(p)}{n}}\right).
	\]
\end{lemma}

\begin{proof}
	We have 
	\begin{align*}
	& \infMnorm{\SMI-\SPMI}\\
	& \le\sup_{\ell,r\in[p]}\frac{1}{n}\sum_{i=1}^{n}\left|\frac{\gamma^2_{ik}(\hbtheta)}{\hsigma_{k}^{2}}\frac{\left(y_{i}-\x_{i}^{T}\hbbeta_{k}\right)^{2}}{\hsigma_{k}^{2}}-\frac{\gamma^2_{ik}(\btheta)}{\sigma_{k}^{2}}\frac{\left(y_{i}-\x_{i}^{T}\bbeta_{k}\right)^{2}}{\sigma_{k}^{2}}\right|\left|x_{i,\ell}\right|\left|x_{i,r}\right|\\
	& \le\left(\sup_{i\in[n]}\left|\frac{\gamma^2_{ik}(\hbtheta)}{\hsigma_{k}^{2}}\frac{\left(y_{i}-\x_{i}^{T}\hbbeta_{k}\right)^{2}}{\hsigma_{k}^{2}}-\frac{\gamma^2_{ik}(\btheta)}{\sigma_{k}^{2}}\frac{\left(y_{i}-\x_{i}^{T}\bbeta_{k}\right)^{2}}{\sigma_{k}^{2}}\right|\right)\left(\sup_{\ell,r\in[p]}\frac{1}{n}\sum_{i=1}^{n}\left|x_{i,\ell}\right|\left|x_{i,r}\right|\right)\\
	& \le\left(\underbrace{\sup_{i\in[n]}\left|\frac{\gamma^2_{ik}(\hbtheta)}{\sigma_{k}^{4}} \left(y_{i}-\x_{i}^{T}\hbbeta_{k}\right)^{2}\right|\left|1-\frac{\sigma_{k}^{4}}{\hsigma_{k}^{4}}\right|}_{(a)}+\underbrace{\sup_{i\in[n]}\frac{1}{{\sigma_{k}^{4}}} \left|{\gamma^2_{ik}(\hbtheta)\left(y_{i}-\x_{i}^{T}\hbbeta_{k}\right)^{2}-\gamma^2_{ik}(\btheta)\left(y_{i}-\x_{i}^{T}\bbeta_{k}\right)^{2}}\right|}_{(b)}\right)\\
	&\qquad\ \ \cdot \underbrace{\left(\sup_{\ell,r\in[p]}\frac{1}{n}\sum_{i=1}^{n}\left|x_{i,\ell}\right|\left|x_{i,r}\right|\right)}_{(c)}.
	\end{align*}
	
	On the event $\event$ we have that
	\begin{align}
	\sup_{i\in[n],k\in[K]} |y_i - \x_i^T\hbbeta_k| &\le 
	\sup_{i\in[n],k\in[K]}|y_i - \x_i^T\bbeta_k| + \sup_{i\in[n],k\in[K]}
	|\x_i^T(\bbeta_k - \hbbeta_k)|\nonumber\\
	&\le C\sqrt{\log(nK)} + (\sup_{i\in[n]} \infnorm{\x_i}) \left(\sup_{k\in[K]}
	\onenorm{\bbeta_k -\hbbeta_k}\right)\nonumber\\
	&\le C\left(\sqrt{\log(nK)} + \sqrt{\log(np) }\est\right) \nonumber\\
	&= O\left(\sqrt{\log(np)}\right)\,.\label{eq:resid}
	\end{align}
	Therefore part $(a)$ can be bounded as follows:
	\[
	\sup_{i\in[n]}\left|\frac{\gamma^2_{ik}(\hbtheta)\left(y_{i}-\x_{i}^{T}\hbbeta_{k}\right)^{2}}{\sigma_{k}^{4}}\right|\left|1-\frac{\sigma_{k}^{4}}{\hsigma_{k}^{4}}\right|\le 
	 \frac{1}{\sigma_k^4}\left|1-\frac{\sigma_{k}^{4}}{\hsigma_{k}^{4}}\right| \sup_{i\in[n]} \left(y_i - \x_i^T \hbbeta_k\right)^2 = O_{p}\left(\est \log(np)\right),
	\]
	where we used that $\gamma_{i,k}(\hbtheta) \le 1$ and $|\hsigma_k - \sigma_k| = O_p(\est)$ per Assumption (A1). 
	
	 For $(b)$, we  write
	 \begin{align*}
	&\sup_{i\in[n]}\frac{1}{\sigma_{k}^{4}}\left|{\gamma^2_{ik}(\hbtheta)\left(y_{i}-\x_{i}^{T}\hbbeta_{k}\right)^{2}-\gamma^2_{ik}(\btheta)\left(y_{i}-\x_{i}^{T}\bbeta_{k}\right)^{2}}\right| \\
	& = \sup_{i\in[n]} \frac{1}{\sigma_{k}^{4}}\left|{\gamma^2_{ik}(\hbtheta)\left(y_{i}-\x_{i}^{T}\hbbeta_{k}\right)^{2} -\gamma^2_{ik}(\hbtheta)\left(y_{i}-\x_{i}^{T}\bbeta_{k}\right)^{2} + \gamma^2_{ik}(\hbtheta)\left(y_{i}-\x_{i}^{T}\bbeta_{k}\right)^{2} -\gamma^2_{ik}(\btheta)\left(y_{i}-\x_{i}^{T}\bbeta_{k}\right)^{2}}\right| \\
	& \le \sup_{i\in[n]} \frac{1}{\sigma_k^4} \left| (\gamma^2_{ik}(\hbtheta) -\gamma^2_{ik}(\btheta)) \left(y_{i}-\x_{i}^{T}\bbeta_{k}\right)^{2}\right| + \sup_{i\in[n]}\frac{1}{\sigma_{k}^{4}}\left|{\gamma^2_{ik}(\hbtheta)\left(y_{i}-\x_{i}^{T}\hbbeta_{k}\right)^{2} -\gamma^2_{ik}(\hbtheta)\left(y_{i}-\x_{i}^{T}\bbeta_{k}\right)^{2} }\right| \\
	&=\sup_{i\in[n]} \frac{1}{\sigma_{k}^{4}}\left|{\left(\gamma^2_{ik}(\hbtheta)-\gamma^2_{ik}(\btheta)\right)\left(y_{i}-\x_{i}^{T}\bbeta_{k}\right)^{2}}\right|+\sup_{i\in[n]}\frac{1}{\sigma_k^4}\left|\gamma^2_{ik}(\hbtheta)\left(\x_{i}^{T}\left(\hbbeta_{k}-\bbeta_{k}\right)\right)\left(y_{i}-\x_{i}^{T}\bbeta_{k}+y_{i}-\x_{i}^{T}\hbbeta_{k}\right)\right|\\
	& \le\sup_{i\in[n]}\frac{2}{\sigma_k^4} |\gamma_{ik}(\hbtheta)-\gamma_{ik}(\btheta)| \left|y_i-\x_i^T\bbeta_k\right|^2+\sup_{i\in[n]}\frac{1}{\sigma_k^4} \left|\x_{i}^{T}\left(\hbbeta_{k}-\bbeta_{k}\right)\right| \left|y_{i}-\x_{i}^{T}\bbeta_{k}+y_{i}-\x_{i}^{T}\hbbeta_{k}\right|\\
	& \le\sup_{i\in[n]}\frac{2}{\sigma_k^4} |\gamma_{ik}(\hbtheta)-\gamma_{ik}(\btheta)| \left|y_i-\x_i^T\bbeta_k\right|^2+\sup_{i\in[n]}\frac{1}{\sigma_k^4} \infnorm{\x_i} \onenorm{\hbbeta_{k}-\bbeta_{k}} \left|y_{i}-\x_{i}^{T}\bbeta_{k}+y_{i}-\x_{i}^{T}\hbbeta_{k}\right|\\
	&= O(\log^{3/2}(np) \sqrt{\log(nK)} \est) + O(\sqrt{\log(nK)}\sqrt{\log(np)}\est)\nonumber\\
	&= O(\log^{3/2}(np) \sqrt{\log(nK)} \est) \,,
	\end{align*}
	where in the penultimate step, we bounded the first term using
        Lemma \ref{lemma:lipschitz_gamma} together with Assumption
        (A1) and we bounded the second term using definition of event
        $\event$ along with~\eqref{eq:resid}. 
	
Finally, for $(c)$, we use
	Lemma \ref{lem:(Concentration-of-Design} with $\bxi= \e_r$ for $r\in[p]$ followed by a union bound over $r$ to get
	\[
	\sup_{\ell,r\in[p]}\frac{1}{n}\sum_{i=1}^{n}\left|x_{i,\ell}\right|\left|x_{i,r}\right|=O_{p}(1).
	\]
	Putting $(a)$, $(b)$, and $(c)$ together, we have
	\[
	\infMnorm{\SMI-\SPMI}=O_p(\log^{3/2}(np)\sqrt{\log(nK)} \est) .
	\]
	
	To prove the second result, we have
	\begin{align}
	\P\left(\infMnorm{\SPMI-\PMI}\ge t\right) 
	&\le \P\left(\infMnorm{\SPMI- \PMI}\ge t;\event\right) + \P(\event^c)\nonumber\\
	&\le \P\left(\infMnorm{\SPMI-\PMI}\ge t;\event\right) + \frac{4}{n^2}\,,\label{eq:event-c-1}
	\end{align}
	where we used the result of Lemma~\ref{lem:events} to bound $\P(\event^c)$.
	We next write
	\begin{align}
	& \P\left(\infMnorm{\SPMI-\PMI}\ge t;\event \right)\nonumber\\
	& = \P\left(\sup_{\ell,r\in[p]}\left|\frac{1}{n}\sum_{i=1}^{n}\frac{\gamma^2_{ik}(\btheta)}{\sigma_{k}^{4}}{\left(y_{i}-\x_{i}^{T}\bbeta_{k}\right)^{2}}x_{i,\ell}x_{i,r}-\mathbb{E}\left[\frac{\gamma^2_{ik}(\btheta)}{\sigma_{k}^{4}}{\left(y_{i}-\x_{i}^{T}\bbeta_{k}\right)^{2}}x_{i,\ell}x_{i,r}\right]\right|\ge t;\event\right)\nonumber \\
	&\le \P\left(\Big(\sup_{i\in[n]} (y_i-\x_i^T\bbeta_k)^2\Big)\cdot\sup_{\ell,r\in[p]} \frac{1}{n}  \sum_{i=1}^{n}\left|\frac{\gamma^2_{ik}(\btheta)}{\sigma_{k}^{4}} x_{i,\ell}x_{i,r}-\mathbb{E}\left[\frac{\gamma^2_{ik}(\btheta)}{\sigma_{k}^{4}}x_{i,\ell}x_{i,r}\right]\right|\ge t;\event\right)\nonumber\\
	&\le \P\left(\sup_{\ell,r\in[p]} \frac{1}{n}  \sum_{i=1}^{n}\left|\frac{\gamma^2_{ik}(\btheta)}{\sigma_{k}^{4}} x_{i,\ell}x_{i,r}-\mathbb{E}\left[\frac{\gamma^2_{ik}(\btheta)}{\sigma_{k}^{4}}x_{i,\ell}x_{i,r}\right]\right|\ge \frac{t}{C^2\log(np)};\event\right)\nonumber\\
	&\le \P\left(\sup_{\ell,r\in[p]} \frac{1}{n}  \sum_{i=1}^{n}\left|\frac{\gamma^2_{ik}(\btheta)}{\sigma_{k}^{4}} x_{i,\ell}x_{i,r}-\mathbb{E}\left[\frac{\gamma^2_{ik}(\btheta)}{\sigma_{k}^{4}}x_{i,\ell}x_{i,r}\right]\right|\ge \frac{t}{C^2\log(np)}\right)\,,\label{eq:event-c-2}
	\end{align}
	where the second last step follows from definition of event $\event$.
	
	Now by using equation~\eqref{eq:sub-exp-1} with $\bxi = \e_r$, we have that $x_{i,\ell}x_{i,r}$ is sub-exponential and since
	 $\frac{\gamma^2_{ik}(\btheta)}{\sigma_{k}^{2}}$
	is  bounded, we get that $\frac{\gamma^2_{ik}(\btheta)}{\sigma_{k}^{4}}x_{i,\ell}x_{i,r}$
	is a sub-exponential random variable. Thus, by Bernstein's inequality (see e.g.~\cite[Theorem 2.8.1]{vershynin2018high})
	and union bound over the $p^{2}$ choices of $\ell,r$, we obtain,
	\[
	\P\left(\sup_{\ell,r\in[p]} \frac{1}{n}  \sum_{i=1}^{n}\left|\frac{\gamma^2_{ik}(\btheta)}{\sigma_{k}^{4}} x_{i,\ell}x_{i,r}-\mathbb{E}\left[\frac{\gamma^2_{ik}(\btheta)}{\sigma_{k}^{4}}x_{i,\ell}x_{i,r}\right]\right|\ge \frac{t}{C^2\log(np)}\right) \le 2p^2 \exp\left(-c \frac{nt^2}{\log^2(np)}\right)\,.
	\]
	Therefore by choosing $t = 2\sqrt{\frac{\log(p)}{cn}}\log(np)$ we obtain
	\begin{align}
	\P\left(\sup_{\ell,r\in[p]} \frac{1}{n}  \sum_{i=1}^{n}\left|\frac{\gamma^2_{ik}(\btheta)}{\sigma_{k}^{4}} x_{i,\ell}x_{i,r}-\mathbb{E}\left[\frac{\gamma^2_{ik}(\btheta)}{\sigma_{k}^{4}}x_{i,\ell}x_{i,r}\right]\right|\ge \frac{2}{C^2} \sqrt{\frac{\log(p)}{cn}}\right) \le 2p^{-2}\,.\label{eq:event-c-3}
	\end{align}
\end{proof}

Combining equations~\eqref{eq:event-c-1}, \eqref{eq:event-c-2} and \eqref{eq:event-c-3} we arrive at

\[
\infMnorm{\SPMI- \PMI} = O_p\left(\log(np)\sqrt{\frac{\log(p)}{n}}\right)\,,
\]
which completes the proof of the second claim.

\bibliographystyle{agsm}
\bibliography{refs}	
\end{document}